\documentclass[]{amsart}

\usepackage[margin=1in]{geometry}
\usepackage{amsmath, amsthm, amssymb, mathrsfs, enumitem, mathtools, tikz-cd, setspace, hyperref, float}

\usepackage[backend=biber, style=numeric]{biblatex}

\usepackage[T1]{fontenc}
\usepackage{mlmodern}

\setlist[enumerate,1]{label={(\alph*)}}

\allowdisplaybreaks
\newcommand{\Z}{\mathbb{Z}}

\newcommand{\C}{\mathbb{C}}
\newcommand{\Q}{\mathbb{Q}}
\newcommand{\Qp}{{\mathbb{Q}_p}}

\newcommand{\Qpf}{{\mathbb{Q}_{p^f}}}

\newcommand{\Zp}{{\mathbb{Z}_p}}
\newcommand{\Fpf}{{\mathbb{F}_{p^f}}}

\newcommand{\B}{\mathbf{B}}

\newcommand{\D}{\mathbf{D}}

\newcommand{\V}{\mathbf{V}}

\newcommand{\F}{\mathbb{F}}
\newcommand{\ident}{\mathbb{I}}
\newcommand{\Ms}{\mathcal{M}}
\newcommand{\Mf}{\mathfrak{M}}
\newcommand{\M}{\mathbf{M}}

\newcommand{\pf}{\mathfrak{p}}

\newcommand{\pu}{[\![u]\!]}

\newcommand{\eval}{\text{\normalfont eval}}
\newcommand{\Kis}{\text{\normalfont Kis}}
\newcommand{\Bre}{\text{\normalfont Bre}}
\newcommand{\Rep}{\text{\normalfont Rep}}

\newcommand{\ind}{\text{\normalfont Ind}}
\newcommand{\semi}{\text{\normalfont ss}}
\newcommand{\et}{\text{\normalfont \'{e}t}}

\newcommand{\Fil}{\text{\normalfont Fil}}
\newcommand{\I}{\text{\normalfont I}}
\newcommand{\II}{\text{\normalfont II}}

\newcommand{\Rank}{\text{\normalfont Rank}}

\newcommand{\Diag}{\text{\normalfont Diag}}
\newcommand{\Gal}{\text{\normalfont Gal}}
\newcommand{\GL}{\text{\normalfont GL}}
\newcommand{\SL}{\text{\normalfont SL}}
\newcommand{\Hom}{\text{\normalfont Hom}}
\newcommand{\HP}{\text{\normalfont HP}}

\newcommand{\Mod}{\text{\normalfont Mod}}
\newcommand{\MF}{\text{\normalfont MF}}
\newcommand{\SD}{\text{\normalfont SD}}

\newcommand{\Mat}{\text{\normalfont Mat}}

\newcommand{\cris}{\text{\normalfont cris}}
\newcommand{\st}{\text{\normalfont st}}

\newcommand{\BK}{\text{\normalfont BK}}

\newcommand{\res}{\text{\normalfont res}}

\newcommand{\oh}{\mathscr O}
\newcommand{\Para}{\mathscr{P}}
\newcommand{\Ts}{\mathcal T}
\newcommand{\Ss}{\mathcal S}

\newcommand{\Sig}{\mathfrak{S}}
\newcommand{\idem}{\mathfrak{e}}

\newcommand{\kvec}{\mathbf{k}}
\newcommand{\As}{\mathcal{A}}
\newcommand{\Af}{\mathfrak{A}}
\newcommand{\Bs}{\mathcal{B}}

\newcommand{\Is}{\mathcal{I}}
\newcommand{\Ps}{\mathcal{P}}

\newcommand{\Os}{\mathcal{O}}
\newcommand{\Ds}{\mathcal{D}}

\newtheoremstyle{italics}{}{}{\itshape}{}{\bfseries}{:}{ }{}
\theoremstyle{italics}
\newtheorem{thm}[subsubsection]{Theorem}
\newtheorem{lem}[subsubsection]{Lemma}
\newtheorem{prop}[subsubsection]{Proposition}

\newtheoremstyle{noitalics}{}{}{}{}{\bfseries}{:}{ }{}
\theoremstyle{noitalics}
\newtheorem{rem}[subsubsection]{Remark}

\newtheorem{defin}[subsubsection]{Definition}
\newtheorem{assump}[subsubsection]{Assumption}
\newtheorem*{question}{Question}

\theoremstyle{italics}
\newtheorem{thmx}{Theorem}

\title{Reductions of Some Crystalline Representations in the Unramified Setting}
\author{Anthony Guzman}
\address{Department of Mathematics, University of Arizona, Tucson, AZ 85721 USA}
\email{awguzman@arizona.edu}
\date{\today}

\begin{document}
	
\begin{abstract}
	We determine semisimple reductions of irreducible, 2-dimensional crystalline representations of the absolute Galois group $\Gal(\overline{\Q}_p/\Qpf)$. To this end, we provide explicit representatives for the isomorphism classes of the associated weakly admissible filtered $\varphi$-modules by concretely describing the strongly divisible lattices which characterize the structure of the aforementioned modules. Using these representatives, we construct Kisin modules canonically associated to Galois stable lattice representations inside our crystalline representations. This allows us to compute the reduction of such crystalline representations for arbitrary labeled Hodge-Tate weights so long as the $p$-adic valuations of certain parameters are sufficiently large. Hence, we provide a Berger-Li-Zhu type bound in the unramified setting.
\end{abstract}
	
\maketitle

\setcounter{tocdepth}{1}
\tableofcontents

\section{Introduction}\label{intro-sec}

\subsection{Motivation}
Let $g=\sum_{n\ge1} a_nq^n$ be a weight $k\ge2$ normalized cuspidal eigenform for $\Gamma(N)\subseteq\SL_2(\Z)$ for some level $N\ge1$. The work of Deligne \cite{De71} allows us to attach to $g$, a continuous, two-dimensional $p$-adic representation \[\rho_g:G_\Q\rightarrow\GL_2(\overline{\Q}_p)\] of $G_\Q=\Gal(\overline{\Q}/\Q)$ for any prime $p$. By fixing an embedding of algebraic closures $\overline{\Q}\hookrightarrow\overline{\Q}_p$, we may choose a place of $\overline{\Q}$ above $p$ for which the decomposition group $D_p$ at this place is isomorphic to the local absolute Galois group $G_{\Qp}=\Gal(\overline{\Q}_p/\Qp)$. By restricting $\rho_g$ to $D_p$, we give rise to a local representation of $G_\Qp$, \[\rho_{g,p}:G_{\Qp}\rightarrow\GL_2(\overline{\Q}_p).\] 
If we assume that $\nu_p(a_p)>0$ with $a_p^2\neq 4p^{k-1}$ then $\rho_{g,p}$ is irreducible and crystalline with Hodge-Tate weights $\{0,k-1\}$ by the works of Faltings and Scholl in \cite{Fal89} and \cite{Sch90}. In \cite{Br03}, Breuil shows that the data of the Hecke eigenvalue $a_p$ and the weight $k$ suffices to completely determine the associated filtered $\varphi$-module $\D_\cris^*(\rho_{g,p})=D_{k,a_p}$ up to a twist by an unramified character. In particular, by choosing a suitable basis $\{\eta_1,\eta_2\}$, one is able to describe the Frobenius and filtration structures on $D_{k,a_p}$ by
\begin{align}\label{Qp-class-equ}
	[\varphi]_\eta&=\begin{pmatrix}
		0 & -1 \\ p^{k-1} & a_p
	\end{pmatrix}&\Fil^j D_{k,a_p}&=\begin{dcases}
		D_{k,a_p} & 0\ge j \\
		\Q_p (\eta_1) & 0<j\le k-1 \\
		0 & k-1< j.
	\end{dcases}
\end{align}

\noindent In fact, every irreducible, two-dimensional crystalline $\Qp$-representation  of $G_\Qp$ takes the form $V_{k,a_p}\coloneqq \V_\cris^*(D_{k,a_p})$, up to the aforementioned twist, for some $a_p$ with $\nu_p(a_p)>0$ and Hodge-Tate weight $k\ge 2$ so that we have an isomorphism $\rho_{g,p}\cong V_{k,a_p}$.

The above discussion lends merit to the idea that the local study of $p$-adic Galois representations may help to provide answers to global questions regarding modular forms. In particular, the residual representations $\overline{\rho}_{g,p}:G_{\Qp}\rightarrow\GL_2(\overline{\F}_p)$ have proven to be especially fruitful, having played an important part in the proof of Serre's modularity conjectures by Khare and Wintenberger in \cite{KW09a, KW09b}. However, our interest is in the following purely local question:

\begin{question}
	What is the isomorphism class of $\overline{V}_{k,a_p}$ as an explicit function of the parameters $k$ and $a_p$?
\end{question}

\noindent The study of this question has proven to be far more complicated than in characteristic zero setting. The essential problem is that in order to determine $\overline{V}_{k,a_p}$, one must first describe an integral structure associated to $D_{k,a_p}$. Let us recall one pursuit toward providing answers to this question when the $p$-adic valuation of $a_p$ is large\footnote{For a summary of reductions when the valuation of $a_p$ is small, see \cite[Thm 5.2.1]{Ber10}.}.

\begin{thm}\label{Qp-large-val-red-thm}
	The isomorphism $\overline{V}_{k,a_p}\cong \overline{V}_{k,0}$ is known in the following increasingly inclusive cases:
	\begin{enumerate}
		\item (Berger-Li-Zhu, \cite{BLZ04}) Whenever $\nu_p(a_p)>\lfloor\frac{k-2}{p-1}\rfloor$;
		\item (Bergdall-Levin, \cite{BL20}) Whenever $\nu_p(a_p)>\lfloor\frac{k-1}{p}\rfloor$ for $p<2$;
		\item (Arsovski, \cite{Ars21}) Whenever $\nu_p(a_p)>\lfloor \frac{k-1}{p+1}\rfloor+\lfloor\log_p(k-1)\rfloor$ for $p>3$ and $p+1\nmid k-1$.
	\end{enumerate}
\end{thm}

\noindent Hence, the isomorphism class of $\overline{V}_{k,0}$ is explicitly determined by the weight $k$\footnote{Theorem \ref{Qp-large-val-red-thm} may be thought of as a collection of results toward the conjectured bound of $\nu_p(a_p)>\lfloor \frac{k-1}{p+1}\rfloor$ inspired by global considerations about modular forms of Gouv\^{e}a in \cite{Gou01}}.  

In this article, we aim to provide an analogous bound in the case that modular forms are replaced with Hilbert modular forms. The key difference at this generality on the representation theoretic side is the replacement of the base field $\Qp$ with the unramified extension $\Qpf$ where $f$ denotes the degree of the totally real number field for which the Hilbert modular form is defined over. For details on the global side of things, see the introduction of \cite{Dou10}.

The study of reductions of such representations is far less advanced than in the $G_\Qp$ case detailed in \ref{Qp-large-val-red-thm}. The initial obstruction is that we lose the nice classification of filtered $\varphi$-modules $\D_\cris^*(\rho_{g,\pf})$ as we had over $\Qp$ in \ref{Qp-class-equ}. Indeed, the data of the Hecke eigenvalue of $g$ at $\pf$ and the weights no longer suffice to describe the structure of $\D_\cris^*(\rho_{g,\pf})$. In essence, there are not enough invariants provided by the Hilbert modular form to concretely determine the associated weakly admissible filtered $\varphi$-module. In order to compute these reductions, we must first understand the required parameters by formulating models for (the isomorphism classes of) weakly admissible filtered $\varphi$-modules corresponding to irreducible representations.

Some partial results are known however. Notably, the work of Dousmanis in \cite{Dou09} gives such a description of filtered $\varphi$-modules which is used in \cite{Dou10} and \cite{Dou13} to generalize the results of \cite{BLZ04} and \cite{Ber11} to compute reductions for certain (infinite) families of crystalline representations whose Frobenius action `looks like' the Frobenius matrix in \ref{Qp-class-equ}. Our results can be viewed as an improvement of these works.

\subsection{Notation and Conventions}\label{notation-sec}

We fix the following notation for the remainder of this article. Let $p\ge2$ be a prime and equip $\Qp$ with a $p$-adic valuation $\nu_p$ normalized so that $\nu_p(p)=1$.\footnote{We exclude $p=2$ in the second part for technical reasons. See the Descent Algorithm \ref{red-alg-prop}.} Upon fixing an algebraic closure $\overline{\Q}_p$, we denote the $p$-adic completion of $\overline{\Q}_p$ by $\C_p$. 

Let $K$ be the unique unramified extension of $\Qp$ with (inertial) degree $f\ge1$; that is, we have an isomorphism $K\cong \Qpf$. Denote its ring of integers by $\oh_K$ with uniformizer $\pi_K$ and its residue field $k=\oh_K/\pi_K\oh_K\cong\Fpf$. Let $\sigma_K$ denote the absolute Frobenius on $K$ induced from the natural Frobenius on the residue field $k$. We define the absolute Galois group of $K$ to be $G_K=\Gal(\overline{K}/K)$ and we will let $I_K$ denote the inertia subgroup of $G_K$.

\begin{rem}\label{K=K0-rem}
	For a general finite extension $K/\Qp$, we define $K_0=W(k)[1/p]$ where $W(k)$ is the ring of Witt vectors over the (perfect) residue field $k$. It follows that $K_0/\Q_p$ is the maximal unramified extension of $\Q_p$ contained in $K$. Since we are taking $K$ to be unramified, then $K_0=K$ and $\oh_K=W(k)$; however, the field $K_0$ occurs naturally in much of $p$-adic Hodge theory so we will sometimes write $K_0$ when it is relevant to the theory with the understanding that we are actually talking about $K$. We hope this causes no confusion.
\end{rem}

Fix a uniformizer $\pi_K=-p$ of $K$ so that the Eisenstein polynomial of $\pi_K$ in the formal variable $u$ may be denoted $E\coloneqq E(u)=u+p\in\oh_K[u]$. We fix once and for all a $p$-power compatible sequence $\pi=(\pi_0,\pi_1,\pi_2,\dots)$ in $\overline{K}$ such that $\pi_0=\pi_K$ and $\pi_n^p=\pi_{n-1}$. We define the field $K_\infty$ to be the compositum $K_\infty=\cup_nK(\pi_n)$ contained in $\overline{K}$ and the absolute Galois group of $K_\infty$ is denoted $G_\infty\coloneqq\Gal(\overline{K}/K_\infty)$. 

Let $\Lambda\subset K_0\pu$ denote the ring of rigid analytic functions on the open $p$-adic unit disc in $K$ and set $\Sig\coloneqq W(k)\pu\subset \Lambda$. The ring $K_0\pu$ admits a Frobenius action $\varphi$ which acts on coefficients by the absolute Frobenius $\sigma_K$ on $K$ and acts on the formal variable $u$ by $\varphi(u)=u^p$. Observe that the rings $\Sig\subset\Lambda\subset K_0\pu$ are $\varphi$-stable under this definition.

Acting as linear coefficients, let $F$ be a finite extension of $\Qp$ taken large enough to admit an embedding $\tau_0:K\hookrightarrow F$. Fix a uniformizer $\varpi$ of $F$ and define its ring of integers by $\oh_F$ with residue field $k_F\coloneqq\oh_F/\varpi\oh_F$. Extending the scalars of our rigid analytic functions, let $\Sig_F\coloneqq \Sig\otimes_{\Zp}\oh_F$ and $\Lambda_F\coloneqq \Lambda\otimes_{\Q_p}F$. Extending the $\varphi$-action on $K_0\pu$ by $F$-linearity, we obtain $\varphi$-stable rings $\Sig_F\subset\Lambda_F\subset (F\otimes_{\Q_p}K_0)\pu$. 

\subsection{Main Results}\label{overview-sec}

Suppose $V$ is a $d$-dimensional crystalline $F$-representation of $G_K$. In Section \ref{prelim-chap} we will recall how to naturally attach to $V$, a unique weakly admissible filtered $\varphi$-module $D\coloneqq\D_\cris^*(V)$ over $F\otimes_{\Q_p}K$ and how this attachment gives us a natural correspondence between the isomorphism classes of such objects. Since $K$ is an unramified extension of $\Qp$ of degree $f>1$, the structure of $D$ over the tensor $F\otimes_{\Q_p}K$ is complicated and hence, so are the isomorphism classes. To get a hold of this structure, we utilize the $f$-embeddings of $K\hookrightarrow F$ to decompose $D=\prod D^{(i)}$ into pieces, each being viewed as filtered $\varphi$-modules over $F$, so that we can write down explicit partial Frobenius matrices and filtration structures on each individual piece. In essence, we choose to trade a mysterious single $D$ for $f$-many $D^{(i)}$, each of which we aim to understand completely.

The weak admissibility property on our filtered $\varphi$-modules can be used to further restrict the structure on $D$ and hence, on each $D^{(i)}$. Indeed, we utilize the results of \cite{Zhu09} in Section \ref{SDL-chap} to prove the existence of a strongly divisible lattice $L$ inside $D$ and detail how the isomorphism classes of strongly divisible lattices determine the structure of weakly admissible filtered $\varphi$-modules. This strongly divisible structure will descend along the decomposition $D=\prod D^{(i)}$ so that we may write $L=\prod L^{(i)}$ where each $L^{(i)}$ may be viewed as a strongly divisible lattice inside $D^{(i)}$.

We then choose a basis so that the filtration structure is fixed for our strongly divisible lattice $L$, and hence it will be fixed on $D$. In this basis, the structure of $L$ is completely determined by a $f$-tuple of matrices $(A)=(A^{(0)},A^{(1)},\dots, A^{(f-1)})\in\GL_d(\oh_F)^f$ in that we write $L=L(A)=\prod L(A)^{(i)}$ where the partial Frobenius matrix on $L(A)^{(i)}$ in this chosen basis is determined by $A^{(i)}\in\GL_2(\oh_F)$. The result is that the isomorphism classes of strongly divisible lattices are determined by the $f$-tuple $(A^{(i)})\in\GL_d(\oh_F)^f$. Moreover, we give a concrete method to describe such isomorphism classes in terms of parabolic equivalence classes on $\GL_d(\oh_F)^f$ via a bijection $\Theta:[(A)]\leftrightarrow L(A)$. 

In Section \ref{models-irred-reps-chap} we use, in the dimension two case, a matrix simplifying algorithm inspired by similar techniques in \cite{Liu21} to find a `nice' representative of $[(A)]$ in $\GL_2(\oh_F)^f$ so that each $A^{(i)}$ takes the form of one of two matrix \textit{Types}:
\begin{itemize}
	\item Type $\I$:	$A^{(i)}=\begin{pmatrix}0 & a_{1}^{(i)} \\ 1 & a_{2}^{(i)} \end{pmatrix}$ where $a_{1}^{(i)}\in\oh_F^*$ and $a_{2}^{(i)}\in\oh_F$.
	\item Type $\II$:	$A^{(i)}=\begin{pmatrix}a_{1}^{(i)} & 0 \\ a_{2}^{(i)} & 1\end{pmatrix}$ where $a_{1}^{(i)}\in\oh_F^*$ and $a_{2}^{(i)}\in\varpi\oh_F$.
\end{itemize}

\noindent Since the equivalence class $[(A)]$ determines an isomorphism class of strongly divisible lattices which in turn dictate the structure of weakly admissible filtered $\varphi$-modules, then $D(A)\coloneqq L(A)\otimes_{\oh_F}F$ will be a `simple' representative of its isomorphism class. We write $V(A)$ to be the crystalline representation whose image under $\D_\cris^*$ is $D(A)$. The first of our main results is as follows:

\begin{thmx}[Theorem \ref{model-irred-cris-reps-thm}]
	Let $V\in\Rep_{\cris/F}^{\kvec}(G_K)$ be a two dimensional irreducible crystalline $F$-representation of $G_K$ with labeled Hodge-Tate weights $\kvec=(0,k_i)_{i\in\Z/f\Z}$ where $k_i>0$. Then $V\cong V(A)$ where $D(A)=\D^*_\cris(V(A))$ and there exists a basis $\{\eta_1^{(i)},\eta_2^{(i)}\}$ of $D(A)=\prod D(A)^{(i)}$ such that
	\begin{align*}
		[\varphi^{(i)}]_\eta &=\begin{dcases}
			\begin{pmatrix}
				0 & a_1^{(i)} \\ p^{k_{i-1}} & a_2^{(i)}
			\end{pmatrix} & \text{\normalfont Type $\I$} \\
			\begin{pmatrix}
				a_1^{(i)}p^{k_{i-1}} & 0 \\ a_2^{(i)}p^{k_{i-1}} & 1
			\end{pmatrix} & \text{\normalfont Type $\II$}
		\end{dcases} & \Fil^j D(A)^{(i)}&=\begin{dcases}
		D(A)^{(i)} & j\le 0 \\
		F(\eta_1^{(i)}) &  0<j\le k_i \\
		0 & k_i<j
		\end{dcases}
	\end{align*}
\end{thmx}

\noindent In practice, the utility of this theorem is in reducing the number of unrestricted parameters from $4f$ to $2f$ parameters and we additionally are able to place a $p$-adic valuation condition on the $f$-many $a_2^{(i)}$ parameters irreducibility.\footnote{By fixing determinants we may further restrict the defining parameters to just the $f$-many $a_2^{(i)}$ parameters.}

With an understanding of the structure of $V(A)$ in place, we then move on to the task of computing the semi-simple modulo $\varpi$ reductions of irreducible, two-dimensional crystalline representations $V(A)$. By this, we mean the computation of $k_F$-representations $\overline{V(A)}=(T/\varpi T)^{\semi}$ where $T\subset V(A)$ is a $G_K$-stable $\oh_F$-lattice inside of $V(A)$. The study of such a drastic operation lies in the field of \textit{integral} $p$-adic Hodge theory which is significantly more complicated than the rational theory discussed in much of the first part. The main obstruction turns out to be the determination of suitable integral semi-linear algebraic structures associated to $D(A)\coloneqq\D_\cris^*(V(A))$. Unfortunately, the integral data of strongly divisible lattices do not suffice to describe the reductions in the generality that we desire.

For these reasons, we are incentivized to introduce the notion of Breuil and Kisin modules into our discussion. In essence, these modules are defined over power series rings and admit Frobenius and filtration structures. Their main utility is in describing the required semi-linear algebraic data that we need to compute reductions. In particular, a Kisin module $\Mf$ over $\Sig_F$ can be used to compute $\overline{V(A)}$ via the data provided by the modulo $\varpi$ reduction $\overline{\Mf}$ \textit{if} the Kisin module $\Mf$ is canonically associated to a Galois-stable $\oh_F$-lattice $T\subset V(A)$. The issue that persists is that it is very difficult to determine the structure of $\Mf$ from $V(A)$. To get around this, we follow results of \cite{BLL22} in Section \ref{Explicit-Kisin-mod-chap} to detail a way to explicitly construct a Kisin module $\Ms(\As)$ over an enlarged coefficient ring $S_F$ from the weakly admissible filtered $\varphi$-module $D(A)$. Here, we mean to say that we construct an $f$-tuple of matrices $(\As)$ from $(A)$ so that $(\As)$ determines the Frobenius action on $\Ms(\As)$. The largest labeled Hodge-Tate weights $(0,k_i)$ of $D(A)^{(i)}$ then correspond to the labeled heights $k_i$ of the constructed Kisin module $\Ms(\As)^{(i)}$. Since $\Ms(\As)$ is constructed from $D(A)$, then we may use $\Ms(\As)$ to compute the reduction of $V(A)$.

This leaves us with yet another problem. How are we to use the data of $\Ms(\As)$ over $S_F$ to descend the coefficients to an integral Kisin module $\Mf(\Af)$ over $\Sig_F$ so that we can compute the associated reduction? We give an answer in the form of a \textit{descent algorithm} inspired by \cite{Liu21} under some restrictions on height in Section \ref{red-alg-chap}. 

\begin{thmx}[\textit{c.f.} Theorem \ref{descent-thm}]
	Suppose $p>2$. Let $\Ms(\As)$ be a Kisin module over $S_F$ with labeled heights $k_i\le c(p-2)$, for some integer $c>1$ constructed from $D(A)=\D_\cris^*(V(A))$. Suppose there exists a sequence of determinant preserving base changes $(X_n)$ over $S_F[1/p]$ such that by setting $X_n^{(i)}*_\varphi\As^{(i)}=\As_n^{(i)}$, there is a finite $m>0$ such that
	\[\As_m^{(i)}=\Af_0^{(i)}+C^{(i)}\]
	with $\Af_0^{(i)}\in\Mat_2(\Sig_F)$ and $C^{(i)}\in\Mat_2(I_c)$ where $I_c$ is a particular ideal in $S_F$ depending on $c>1$. Then there exists a Kisin module $\Mf(\Af)$ canonically associated to a Galois-stable lattice $T\subset V(A)$ such that $\overline{\Af}^{(i)}\equiv\Af_0^{(i)}\pmod\varpi$.
\end{thmx}

\noindent Hence, supposing we are able to find such a base change $(X_n)$, then the reduction of $V(A)$ may be computed from $\overline{\Mf(\Af)}$ as we understand its structure completely. Our entire strategy can be summarized in the following diagram:
\begin{center}
	\begin{tikzcd}
		V(A) \arrow[rr, "\D_\cris^*"] \arrow[d] &  & D(A) \arrow[rr, "\text{Construct}"] &  & \Ms(\As) \arrow[d, "\text{Descend}"]                               \\
		T \arrow[d]              &  &                                     &  & \Mf(\Af) \arrow[d, "\mod\varpi"] \arrow[llll, leftrightarrow, "\text{Canonical}"'] \\
		\overline{V(A)}        &  &                                     &  & \overline{\Mf(\Af)} \arrow[llll, "\text{Compute}"']                   
	\end{tikzcd}
\end{center}

As an application of the descent algorithm, we then display the calculations necessary to show such a base change exists so long as the $p$-adic valuations of the $a_2^{(i)}$-parameters are sufficiently large. Section \ref{comp-red-chap} then gives us the machinery to compute explicit reductions from $\Mf(\Af)$ so that we arrive at our final main result.

\begin{thmx}[Theorem \ref{large-val-red-thm}]
	For $p>2$, let $V\in\Rep_{\cris/F}^{\kvec}(G_K)$ be an irreducible, two dimensional crystalline representation of $G_K$ with labeled Hodge-Tate weights $\{k_i,0\}_{i\in\Z/f\Z}$ where $k_i>0$. Then $V\cong V(A)$ and if $(A^{(i)})_{i\in\Z/f\Z}$ is such that 
	\[\nu_p(a_2^{(i)})>\max\left\{\Bigg\lfloor\frac{k_i-1}{p-2}\Bigg\rfloor,\Bigg\lfloor\frac{k_{max}-1}{p-2}\Bigg\rfloor-\Bigg\lfloor\frac{k_i-1}{p-2}\Bigg\rfloor-1\right\},\]
	then there exists $v_i,w_i\in\{k_i,0\}$ such that:
	\begin{enumerate}
		\item If $|\Ss|$ is even, then 
		\[\overline{V}|_{I_K}=\omega_f^{\sum_{j=0}^{f-1} p^{j}v_j}\oplus\omega_f^{\sum_{j=0}^{f-1} p^{j}w_j};\]
		\item If $|\Ss|$ is odd, set $t=p\sum_{j=0}^{f-1} p^{j}w_j+\sum_{j=0}^{f-1} p^{j}v_j$, 
		\begin{enumerate}
			\item[(i)] When $p^f-1\nmid t$ then 
			\[\overline{V}|_{G_\infty}=\ind^{G_{\Qpf}}_{G_{\Q_{p^{2f}}}}\left(\omega_{2f}^{t}\right);\]
			
			\item[(ii)] When $p^f-1\mid t$ then 
			\[\overline{V}|_{I_K}=\omega_f^{\frac{t}{p^f-1}}\oplus\omega_f^{\frac{t}{p^f-1}}.\]
		\end{enumerate}
	\end{enumerate}
\end{thmx}

\noindent Hence, we obtain an improvement of \cite{Dou10} where reductions of certain representations are given on the bound $\nu_p(a_2^{(i)})>\lfloor\frac{k_{max}-1}{p-1}\rfloor$ for all $i\in\Z/f\Z$.

\subsection{Acknowledgments} This article alongside \cite{Guz2} were borne out of the authors Ph.D. thesis at the University of Arizona under the direction of Brandon Levin, whose generosity and knowledge made this work possible. It will quickly become evident the influence that Tong Liu and John Bergdall have in this paper, we thank them for the insightful conversations related to this project. 

\section{\texorpdfstring{Preliminaries}{Preliminaries}}\label{prelim-chap}

\subsection{\texorpdfstring{Filtered $\varphi$-Modules and Crystalline Representations}{Filtered phi-Modules and Crystalline Representations}}\label{cris-fil-sec}

Let $\Rep_{/F}(G_K)$ denote the category of $d$-dimensional $F$-representations of $G_K$. For any $V\in\Rep_{/F}(G_K)$, define
\[\D_\cris(V)=(\B_\cris\otimes_{\Q_p}V)^{G_K}.\]
where $\B_\cris$ is Fontaine's crystalline period ring, see \cite{Fon91}. The object $\D_\cris(V)$ is naturally a rank-$d$ filtered $\varphi$-module over $K\otimes_{\Q_p}F$. Note that by Remark \ref{K=K0-rem}, a filtered $\varphi$-module for unramified $K$ will be a $K_0\otimes_{\Q_p}F$-vector space so the filtration and underlying vector space are valued over the same field. We denote the category consisting of such objects by $\MF_{K\otimes_{\Q_p}F}^{\varphi}$ with morphisms being linear maps of filtered vector spaces which are compatible with the $\varphi$-semilinear endomorphisms. In particular, we assume that any morphism $h:D\rightarrow D'$ in this category is \textit{strict} in the sense that $h(\Fil^jD)=h(D)\cap\Fil^jD'$.

Define the composition $\tau_i\coloneqq\tau_0\circ\sigma_K^{-i}:K\hookrightarrow F$ so that the set $|\tau|\coloneqq\{\tau_i\}$ for $0\le i\le f-1$ gives a list of $f$-distinct embeddings of $K$ into $F$. As we have fixed this sequence of embeddings, we will now reserve the index $i\in\Z/f\Z$ to denote the enumeration of this list of embeddings.

Define the set 
\[F^{|\tau|}\coloneqq\prod_{\tau_i:K\hookrightarrow F} F\hspace{.5cm}\text{with elements}\hspace{.5cm} \sum_{x\in F,y\in K}\left(x\tau_0(y),x\tau_1(y),\dots,x\tau_{f-1}(y)\right).\]
The product $F^{|\tau|}$ is a ring with multiplication induced by component wise multiplication on the principle elements given by $(x\tau_i(y))_i\cdot(x'\tau_i(y'))_i\coloneqq (xx'\tau_i(yy'))_i$. This gives a natural ring isomorphism
\begin{align*}
	\theta&:F\otimes_{\Qp}K\rightarrow F^{|\tau|} &x\otimes y&\mapsto (x\tau_i(y))_{i\in\Z/f\Z}.
\end{align*}

\noindent The automorphism $1_F\otimes\sigma_K$ on $F\otimes_{\Qp}K$ transforms via $\theta$ to an automorphism $\varphi:F^{|\tau|}\rightarrow F^{|\tau|}$ defined by $\varphi(v_0,v_1,v_2,\dots,v_{f-1})=(v_1,v_2,\dots,v_{f-1},v_0)$. We endow $F^{|\tau|}$ with an $F\otimes_{\Q_p}K$-module structure defined by $(a\otimes b)(v_0,\dots,v_{f-1})=(a\tau_0(b)v_0,\dots,a\tau_{f-1}(b)v_{f-1})$.

\begin{lem}\label{tensor-decomp-lem}
	The ring isomorphism $\theta:F\otimes_{\Q_p}K\rightarrow F^{|\tau|}$ defines a $\varphi$-equivariant $F\otimes_{\Qp}K$-module isomorphism.
	\begin{proof}
		This map is clearly a module homomorphism. The embeddings $\tau_i$ are linearly independent over $F$ and since $F^{|\tau|}$ can be viewed as an $f$-dimensional $F$-vector space, then the $f$-distinct embeddings $\tau_i$ will define a maximal linearly independent list. Thus the map $\theta$ is an isomorphism of modules. Since the Frobenius action $\varphi$ on $F^{|\tau|}$ is defined by $1_F\otimes \sigma_K$, then it is clear that $\varphi\circ\theta=\theta\circ(1_F\otimes\sigma_K)$ giving $\varphi$-compatibility.
	\end{proof}
\end{lem}

Define $\idem_i\coloneqq(0,\dots,0,1,0,\dots,0)$ to be the idempotent of $\varphi$ in $F^{|\tau|}$ where the $1$ appears in the $i$-th component of the vector. Set $F^{(i)}\coloneqq \idem_iF$ and observe that each $F^{(i)}$ is isomorphic to $F$ as $F$-modules. It follows by Lemma \ref{tensor-decomp-lem} that we may decompose $F\otimes_{\Q_p}K$ in terms of these idempotents
\[F\otimes_{\Qp}K\cong\prod_{i\in\Z/f\Z}F^{(i)}\]
where $\varphi\coloneqq(\varphi^{(0)},\dots,\varphi^{(f-1)})$ with $\varphi^{(i)}:\varphi^*F^{(i-1)}\rightarrow F^{(i)}$, the \textit{i-th partial Frobenius}.

Let $D\in\MF_{F\otimes_{\Q_p}K}^\varphi$ be a filtered $\varphi$-module over $F\otimes_{\Qp}K$. Recall that technically speaking, $D$ is a $F\otimes_{\Q_p}K_0$-vector space but by Remark \ref{K=K0-rem}, $K_0=K$ so we are able to make the following observation. Define $D^{(i)}\coloneqq \idem_iD$ so that the above discussion implies that $\theta:D\rightarrow\prod_i D^{(i)}$ is an isomorphism of $F\otimes_{\Q_p}K$-modules and the Frobenius action on $\prod_i D^{(i)}$ acts via $\varphi_D\coloneqq(\varphi_{D}^{(0)},\dots,\varphi_{D}^{(f-1)})$ with $i$-th partial Frobenius $\varphi_D^{(i)}:\varphi^*D^{(i-1)}\rightarrow D^{(i)}$.

\begin{prop}\label{fil-idem-decomp-prop}
	There is an isomorphism of filtered $\varphi$-modules
	\[D\cong\prod_{i\in\Z/f\Z}D^{(i)}\]
	so that $\Fil^jD^{(i)}=\idem_i\Fil^j D$.
	\begin{proof}
		The fact that $\theta$ is $\varphi_D$-compatible follows from the fact that 
		\[\varphi_D\left(x\tau_{i+1}(y)v\right)_{i\in\Z/f\Z}=\left(x\tau_{i}(y)\varphi_{D}^{(i)}(v)\right)_{i\in\Z/f\Z}\]
		for $x\otimes y\in F\otimes_{\Qp}K$ and $v\in D^{(i+1)}$. Furthermore, we induce a filtration structure on each idempotent piece by taking $\Fil^j D^{(i)}\coloneqq\Fil^jD \cap D^{(i)}=\idem_i\Fil^jD$. This gives us that $\theta$ is compatible with both the Frobenius and filtration structures so we get a morphism of filtered $\varphi$-modules.
	\end{proof}
\end{prop}

With this, we may pass between a filtered $\varphi$-module $D$ and its decomposition $\prod_i D^{(i)}$ which we do so freely by the results of Proposition \ref{fil-idem-decomp-prop} and we can view each piece $D^{(i)}$ as a filtered $\varphi$-module over $F$ via the isomorphism of $F^{(i)}$ and $F$.\footnote{We use `view' loosely here. Indeed, each $D^{(i)}$ will have its own filtration structure but the partial Frobenius are really maps between $D^{(i-1)}$ and $D^{(i)}$.} A morphism in the category $\MF_{F\otimes_{\Qp}K}^\varphi$ can then be interpreted as an $f$-tuple $h\coloneqq(h^{(0)},\dots,h^{(f-1)})$ where each $h^{(i)}:D^{(i)}\rightarrow D'^{(i)}$ is a morphism viewed in the category of filtered $\varphi$-modules over $F$ respecting the idempotent decomposition. 

For each embedding, the integers $k_j\in\Z$ such that $\Fil^{k_j}D^{(i)}\supsetneq\Fil^{k_j+1}D^{(i)}$ are called \textit{labeled Hodge-Tate weights} for the embedding $\tau_i$. Such weights will often be summarized as $\kvec\coloneqq (k_{ij})_{\tau_i}$ where $1\le j\le d$ with weights possibly repeating. 

We denote the proper subcategory of weakly admissible filtered $\varphi$-modules over $K\otimes_{\Q_p}F$ by $\MF_{K\otimes_{\Q_p}F}^{\varphi,w.a.}$ and we denote the proper subcategory of crystalline representations by $\Rep_{\cris/F}(G_K)$.

\begin{thm}\label{CF-equiv-w-coeff-thm}
	There is an equivalence of categories
	\[\D_\cris:\Rep_{\cris/F}(G_K)\rightarrow\MF_{F\otimes_{\Qp}K}^{\varphi,w.a.}\]
	with quasi-inverse functor denoted by $\V_\cris$.
	\begin{proof}
		See \cite[\S~3]{BM02}.
	\end{proof}
\end{thm}

\noindent We note that the above equivalence respects the notion of duality; in other words, we have $\D_\cris(V^*)=(\D_\cris(V))^*\eqcolon \D_\cris^*(V)$ and similarly for $\V_\cris$ so that 
\begin{align*}
	\D_\cris^*(\cdot)&=\Hom_{F[G_K]}(\cdot,\B_\cris) & \V_\cris^*(\cdot)&=\Hom_{\varphi,\Fil}(\cdot,\B_\cris).
\end{align*} 

\noindent Since the functors $\D_\cris$ and $\V_\cris$ are both covariant, it is easy to see that the dual functors will be contravariant. In fact, it is often convenient for the purposes of computing reductions that we use these contravariant functors, see Section \ref{mod-rep-sec}.

We may restrict the equivalence of categories of Theorem \ref{CF-equiv-w-coeff-thm} to crystalline representations with prescribed labeled Hodge-Tate weights $\Rep_{\cris/F}^\kvec(G_K)$ and weakly admissible filtered $\varphi$-modules with the same weights $\MF_{F\otimes_{\Q_p}K}^{\varphi,w.a.,\kvec}$, denoted
\[\D_\cris:\Rep_{\cris/F}^\kvec(G_K)\rightarrow \MF_{F\otimes_{\Q_p}K}^{\varphi,w.a.,\kvec}.\]

\begin{rem}\label{semistable-rem}
	While this article focuses on crystalline representations, certain parts of the theory, see Section \ref{BK-chap}, are more naturally stated in terms of \textit{semistable representations}. At this level, we have an equivalence $\D_\st:\Rep_{\st/F}(G_K)\rightarrow\MF_{F\otimes_{\Q_p}K}^{\varphi,N,w.a.}$ with quasi-inverse $\V_\st$. Here we are taking $\MF_{F\otimes_{\Q_p}K}^{\varphi,N,w.a.}$ to denote the category of weakly admissible filtered $\varphi$-modules admitting a $K_0\otimes_\Qp F$-linear monodromy operator $N$ such that $N\circ\varphi=p\varphi\circ N$. Intuitively, one should think of crystalline representations as being semistable representations with zero monodromy.
\end{rem}

In this article, we are concerned with computing the modulo $p$-reductions of crystalline representations.
\begin{defin}\label{reduction-def}
	For $V\in\Rep_{\cris/F}(G_K)$, let $T\subset V$ be a $G_K$-stable $\oh_F$-lattice. We define the semi-simple modulo $\varpi$ reduction of $V$ to be
	\[\overline{V}=(T\otimes_{\oh_F}k_F)^\semi.\]
\end{defin}
\noindent The existence of such a lattice $T\subset V$ is well known and the reduction is unique up to semi-simplification by the Brauer-Nesbitt theorem. In order to compute $\overline{V}$, we require the following \textit{integral} $p$-adic Hodge Theory.


\subsection{Breuil and Kisin Modules}\label{BK-chap}

Let $R\subseteq(F\otimes_{\Q_p}K_0)\pu$ be a $\varphi$-stable subring containing $E$. We define a \textit{$\varphi$-module} over $R$ to be a finite free $R$-module $M$ admitting an injective $\varphi$-semilinear operator $\varphi_M:M\rightarrow M$.

\begin{defin}\label{kisin-mod-def}
	A \textit{Kisin module over $\Sig_F$ of ($E$-)height $\le h$} is a $\varphi$-module $\Mf$ over $\Sig_F$ such that the linearization map
	\[1\otimes\varphi_\Mf:\varphi^*\Mf\coloneqq \Sig_F\otimes_{\varphi,\Sig_F}\Mf\rightarrow\Mf\]
	has cokernel killed by $E^h$.
\end{defin}

\noindent The category of Kisin modules over $\Sig_F$ with height $\le h$ is denoted $\Mod_{\Sig_F}^{\varphi,\le h}$ with morphisms being $\Sig_F$-linear maps compatible with the Frobenius actions. 

We intend to associate finite height Kisin modules over $\Sig_F$ to Galois-stable $\oh_F$-lattice representations $T\subset V$ inside our crystalline representations. However it is not apparent how a Kisin module as defined arises from a given weakly admissible filtered $\varphi$-module. It turns out that we require an intermediate category which is equivalent to filtered $\varphi$-modules and captures a notion of `rational' Kisin modules. Indeed, by extending scalars we arrive at functor
\[\Mod_{\Sig_F}^{\varphi,\le h}\rightarrow\Mod_{\Lambda_F}^{\varphi,\le h}\]
and a notion of a finite height Kisin module over $\Lambda_F$. We will commonly reserve $\Mf$ to denote a Kisin module over $\Sig_F$ and $\Ms$ to denote a Kisin module over a larger $\varphi$-stable coefficient ring $\Sig_F\subset R$ such as but not limited to $\Lambda_F$. 

\begin{defin}\label{descent-def}
	Let $R$ be a $\varphi$-stable ring such that $\Sig_F\subset R$. When we can write a finite height Kisin module over $R$ by $\Ms=\Mf\otimes_{\Sig_F}R$ then we call $\Mf$ a \textit{descent} of $\Ms$ to $\Sig_F$. 
\end{defin}

By \cite[Proposition~1.2.8]{Kis06}, there is a functor \[\Mod_{\Lambda_F}^{\varphi,\le h}\rightarrow\MF_{F\otimes_{\Qp}K}^{\varphi}.\] 
As topological spaces, the above functor maps $\Ms\mapsto \Ms/u\Ms$ with the induced $\varphi$ action and with filtration structure induced by $\Ms$ being more complicated.\footnote{The formal definition is not used in this article so we will exclude it here, see \cite[Rem 1.2.7]{Kis06} for the proper statement} However, this functor does not define an equivalence as the category $\Mod_{\Lambda_F}^{\varphi,\le h}$ is too large. We instead need to cut out a subcategory inside $\Mod_{\Lambda_F}^{\varphi,\le h}$ to attain the equivalence we require.

\begin{rem}\label{kisin-semistable-rem}
	To capture the notion of \textit{semistable} representations with nonzero monodromy, see Remark \ref{semistable-rem}, we need only add the existence of an $(F\otimes_{\Q_p}K_0)$-linear monodromy operator $N_\Mf$ on $\Mf/u\Mf\otimes_{\oh_F}F$ satisfying $N_\Mf\circ\varphi_{\Mf}=p\varphi_{\Mf}\circ N_{\Mf}$ to Definition \ref{kisin-mod-def}. The category of finite height Kisin modules with monodromy is denoted $\Mod_{\Sig_F}^{\varphi,N,\le h}$. This gives an extension of scalars functor $\Mod_{\Sig_F}^{\varphi,N,\le h}\rightarrow\Mod_{\Lambda_F}^{\varphi,N,\le h}$ and Kisin then describes a functor $\Mod_{\Lambda_F}^{\varphi,N,\le h}\rightarrow\MF_{F\otimes_{\Qp}K}^{\varphi,N}$ into the category of filtered $(\varphi,N)$-modules.
\end{rem}

\begin{defin}\label{kisin-rational-def}
	A finite height $(\varphi,N_\nabla)$-module $\Ms$ over $\Lambda_F$ is a finite height Kisin module over $\Lambda_F$ admitting a differential operator $N_\nabla^{\Ms}$ satisfying the relation \[N_\nabla^\Ms\circ\varphi_\Ms=E\varphi_\Ms\circ N_\nabla^\Ms\] such that $N_\nabla^\Ms|_{u=0}=N_\Ms$.
\end{defin}

\noindent The category of $(\varphi,N_\nabla)$-modules over $\Lambda_F$ of height $\le h$ is denoted $\Mod_{\Lambda_F}^{\varphi,N_\nabla,\le h}$ and it is naturally a proper subcategory $\Mod_{\Lambda_F}^{\varphi,N_\nabla,\le h}\subset\Mod_{\Lambda_F}^{\varphi,N,\le h}$. An effective criteria for the existence of such an operator $N_\nabla$ on a Kisin module $\Ms\in\Mod_{\Lambda_F}^{\varphi,N,\le h}$ is called the \textit{monodromy condition}, see \cite[\S 2.2]{BL20} for details. Restricting to the category of $(\varphi,N_\nabla)$-modules will give us the desired equivalence.

\begin{thm}\label{kisin-equiv-thm}
	There is an equivalence of categories
	\[\D_{\Kis}:\Mod_{\Lambda_F}^{\varphi,N_\nabla}\rightarrow \MF_{F\otimes_{\Qp}K}^{\varphi,N}\]
	with quasi-inverse functor denoted $\M_\Kis$.
	\begin{proof}
		When $F=\Qp$ this is \cite[Theorem~1.2.15]{Kis06} . For general $F$, see \cite[Remark~2.2]{BLL22}.
	\end{proof}
\end{thm}

\noindent Here we have that the functor $\D_\Kis$ defines monodromy on its image by $N=N_\nabla|_{u=0}$. To recover the crystalline case, it follows that we need only consider those $(\varphi,N_\nabla)$-modules where $N_\nabla|_{u=0}=0$.

As hinted at prior, Kisin modules over $\Sig_F$ may be used to compute reductions of crystalline (or more generally semi-stable) representations. This is done via a canonical association of certain Kisin modules to Galois stable $\oh_F$ lattices inside such representations. Let $V\in\Rep_{\cris/F}(G_K)$ be a crystalline representation and recall that $G_\infty=\Gal(\overline{\Q}_p/K_\infty)$. For any $s>0$, we let $\Lambda_{F,s}$ denote the $\Lambda_F$-algebra of rigid analytic functions converging over $|u|<p^{-s}$. The following proposition allows us to make the aforementioned association.

\begin{prop}\label{kis-lattice-ident-prop}
	For $\Mf\in\Mod_{\Sig_F}^{\varphi,\le h}$ and $V\in\Rep_{\cris/F}(G_K)$, suppose there exists $s$ such that $1/p<s<1$ and
	\[\Mf\otimes_{\Sig_F}\Lambda_{F,s}\cong \M_\Kis(\D_\cris^*(V))\otimes_{\Lambda_F}\Lambda_{F,s}\]
	in $\Mod_{\Lambda_{F,s}}^{\varphi,\le h}$. Then $\Mf=\Mf(T)$ is canonically associated to some $G_\infty$-stable $\oh_F$-lattice $T\subset V$.
	\begin{proof}
		See \cite[Prop. 2.1]{BLL22}.
	\end{proof}
\end{prop}

Let $\Os_{\varepsilon,F}$ denote the $p$-adic completion of $\Sig_{F}[1/u]$ with Frobenius action extended from $\Sig_{F}[1/u]$ by continuity. Then $\Os_{\varepsilon,F}$ is a semi-local ring with $\Os_{\varepsilon,F}\otimes_{\oh_F}k_F=k_F(\!(u)\!)$.

\begin{defin}\label{et-phi-mod-def}
	An \textit{\'{e}tale $\varphi$-module} over $\Os_{\varepsilon,F}$ is a finite dimensional $\Os_{\varepsilon,F}$-module $M$ such that:
	\begin{enumerate}
		\item $M$ admits a $\varphi$-semilinear Frobenius endomorphism $\varphi_M:M\rightarrow M$;
		\item The linearization $\varphi_M^*=\Os_{\varepsilon,F}\otimes_{\varphi_M,\Os_{\varepsilon,F}}M\rightarrow M:c\otimes m\mapsto c\varphi_M(m)$ is an isomorphism.
	\end{enumerate}
\end{defin}

\noindent 	We denote the category consisting of such objects by $\Mod_{\Os_{\varepsilon,F}}^{\varphi,\et}$ with morphisms being linear maps compatible with the Frobenius endomorphisms. By replacing $\Os_{\varepsilon,F}$ with $k_F(\!(u)\!)$, we receive the category of \'{e}tale $\varphi$-modules over $k_F(\!(u)\!)$ denoted $\Mod_{k_F(\!(u)\!)}^{\varphi,\et}$. The following theorem of Fontaine shows how the absolute Galois group $G_\infty$ appears in the theory.

\begin{thm}\label{equiv-et-phi-thm}
	There exists covariant equivalence functors
	\begin{align*}
		\V_{\oh_F}&:\Mod_{\Os_{\varepsilon,F}}^{\varphi,\et}\rightarrow\Rep_{/\oh_F}(G_\infty)&\V_{k_F}&:\Mod_{k_F(\!(u)\!)}^{\varphi,\et}\rightarrow\Rep_{/k_F}(G_\infty)
	\end{align*}
	satisfying the following compatibility relation for $M\in\Mod_{\Os_{\varepsilon,F}}^{\varphi,\et}$,
	\[\V_{\oh_F}(M)\otimes_{\oh_F} k_F\cong\V_{k_F}(M\otimes_{\oh_F} k_F).\]
	\begin{proof}
		This is a summary of results from \cite[\S~A1]{Fon91}
	\end{proof}
\end{thm}

Returning to the Kisin module settings, given $\Mf\in\Mod_{\Sig_F}^{\varphi,\le h}$ then $\Mf\otimes_{\Sig_F}\Os_{\varepsilon,F}$ defines an \'{e}tale $\varphi$-module over $\Os_{\varepsilon,F}$ as upon inverting $u$, the polynomial $E(u)$ becomes invertible so the finite height condition on $\Mf$ implies the \'{e}tale condition on $\Mf\otimes_{\Sig_F}\Os_{\varepsilon,F}$. In particular, by the results of Theorem \ref{equiv-et-phi-thm}, we have
\begin{align*}
	\V_{\oh_F}(\Mf\otimes_{\Sig_F}\Os_{\varepsilon,F})&\in\Rep_{/\oh_F}(G_\infty)&\V_{k_F}(\Mf\otimes_{\oh_F}k_F[1/u])&\in\Rep_{/k_F}(G_\infty).
\end{align*}

\begin{prop}\label{det-red-prop}
	Let $\Mf\in\Mod_{\Sig_F}^{\varphi,\le h}$ and $V\in\Rep_{\cris/F}(G_K)$ be such that $\Mf=\Mf(T)$ is the canonically associated Kisin module to a $G_\infty$-stable $\oh_F$-lattice $T\subset V$ from Proposition \ref{kis-lattice-ident-prop}. Then \[(\V_{k_F}^*(\Mf\otimes_{\oh_F}k_F[1/u]))^\semi\cong \overline{V}|_{G_\infty}\]
	\begin{proof}
		By assumption we know that $\overline{\V_\cris^*(\D_\Kis(\Mf\otimes_{\Sig_F}\Lambda_F))}=\overline{V}$ is a semi-simple reduction. On the other hand, since the restriction $\res:\Rep_{\cris/F}(G_K)\rightarrow\Rep_{\cris/F}(G_\infty)$ is a fully faithful functor by \cite[Cor. 2.1.14]{Kis06} then semi-simplicity if preserved so to prove the proposition, we need only show that $\overline{\V_\cris^*(\D_\Kis(\Mf\otimes_{\Sig_F}\Lambda_F))}|_{G_\infty}\cong (\V_{k_F}^*(\Mf\otimes_{\oh_F}k_F[1/u]))^\semi$.
		
		We may write
		\begin{align*}
			\overline{\V_\cris^*(\D_\Kis(\Mf\otimes_{\Sig_F}\Lambda_F))}|_{G_\infty}&=((\V_{\oh_F}^*(\Mf\otimes_{\Sig_F}\Os_{\varepsilon,F})\otimes_{\oh_F}k_F)^\semi \\
			&\cong (\V_{k_F}^*(\Mf\otimes_{\Sig_F}(\Os_{\varepsilon,F}\otimes_{\oh_F}k_F))^\semi \\
			&=(\V_{k_F}^*(\Mf\otimes_{\Sig_F}k_F(\!(u)\!))^\semi
		\end{align*}
		with the isomorphism in the second line following from Theorem \ref{equiv-et-phi-thm}. Since $\Mf$ is a finite height Kisin module then we can write $(\V_{k_F}^*(\Mf\otimes_{\Sig_F}k_F(\!(u)\!))^\semi=(\V_{k_F}^*(\Mf\otimes_{\oh_F}k_F[1/u]))^\semi$ and the proposition follows.
	\end{proof}
\end{prop}

With this, we now have a method to compute reductions of crystalline representations $V$ so long as we can identify the Kisin module associated to some Galois stable $\oh_F$-lattice $T\subset V$. Unfortunately, finding such a Kisin module is generally very difficult. This motivates the need for an auxiliary category that allows us to translate the data of strongly divisible lattices in weakly admissible filtered $\varphi$-modules into Kisin modules. The trade off will be that this can only be done over more complicated coefficients.

Recall that $K_0\pu$ comes equipped with a Frobenius $\varphi$ and define an operator $N=-u\frac{d}{du}$ on $K_0\pu$. Letting $S\coloneqq W(k)[\![u, \frac{E^p}{p}]\!]=\Sig[\![\frac{E^p}{p}]\!]$, we may notice that $S$ is closed under both $\varphi$ and $N$. We also endow $S$ with a filtration structure given by $\Fil^j S=E^j S$. Extending scalars to $F$, define $S_F\coloneqq S\otimes_{\Zp} F$ and extend the the structures of $\varphi$, $N$, and $\Fil^\bullet$ by $F$-linearity. 

\begin{defin}\label{breuil-mod-def}
	A \textit{Breuil module} $\Ds$ over $S_F$ is a $\varphi$-module over $S_F$ such that
	\begin{itemize}
		\item The linearization of $\varphi_{\Ds}$ is an isomorphism;
		\item $\Ds$ is equipped with a decreasing filtration $\Fil^\bullet \Ds$ with $\Fil^0\Ds=\Ds$ and $\Fil^j S_F\cdot\Fil^\ell \Ds\subseteq\Fil^{j+\ell}\Ds$;
		\item $\Ds$ is equipped with a derivation $N_{\Ds}$ over $N$ such that both $N_{\Ds}\varphi_{\Ds}=p\varphi_{\Ds}N_{\Ds}$ and $N_{\Ds}(\Fil^j\Ds)\subseteq \Fil^{j-1}\Ds$.
	\end{itemize}
\end{defin}

\noindent We denote the category of Breuil modules over $S_F$ by $\MF_{S_F}^{\varphi,N}$ with morphisms being $S_F$-linear maps compatible with the $\varphi$, $N$, and $\Fil^\bullet$ structures.

Consider the functor
\[\D_\Bre:\MF_{F\otimes_{\Qp}K}^{\varphi,N}\rightarrow\MF_{S_F}^{\varphi,N}\]
defined by $\D_\Bre(D)\coloneqq S_F\otimes_{K_0\otimes_{\Q_p}F} D=\Ds$ with $\varphi_\Ds=\varphi\otimes\varphi_D$, $N_\Ds=N\otimes 1+1\otimes N_D$ and 
\[\Fil^j\Ds=\{x\in\Ds:N_\Ds(x)\in \Fil^{j-1}\Ds \text{ and }(\eval_\pi\otimes 1)(x)\in\Fil^j D\}\]
where $\eval_\pi:S_F\rightarrow F\otimes_{\Qp}K$ is the scalar extension of the evaluation at $\pi_K$ map $\eval_\pi:W(k)[u]\rightarrow\oh_K$.

\begin{thm}\label{bruel-equiv-thm}
	The functor $\D_\Bre:\MF_{F\otimes_{\Qp}K}^{\varphi,N}\rightarrow\MF_{S_F}^{\varphi,N}$ defines an equivalence of categories.
	\begin{proof}
		See \cite[Theorem 2.3]{BLL22}
	\end{proof}
\end{thm}

Since we are primarily concerned with crystalline representations, we would like to work with Breuil modules without monodromy. However, since the filtration structure is connected to monodromy for a Breuil module, we need a new definition.

\begin{defin}\label{breuil-mod-w/o-mono-def}
	A \textit{Breuil module without monodromy} over $S_F$ is a $\varphi$-module $\Ds$ over $S_F$ such that
	\begin{itemize}
		\item the linearization of $\varphi_\Ds$ is an isomorphism;
		\item There exists $h\ge0$ such that $\Ds$ contains a finite free $S_F$-submodule $\Fil^h\Ds\subseteq\Ds$ such that $\Fil^hS_F\cdot \Ds\subseteq\Fil^h \Ds$. 
	\end{itemize}
\end{defin}

\noindent We denote the category of Breuil modules without monodromy by $\MF_{S_F}^{\varphi,h}$. By \cite[Lemma 2.5]{BLL22} we then have a natural forgetful functor $\MF_{S_F}^{\varphi,N}\rightarrow\MF_{S_F}^{\varphi,h}$ which forgets the monodromy structure.

The results of Theorems \ref{kisin-equiv-thm} and \ref{bruel-equiv-thm} give rise to a threefold equivalence of categories
\[\Mod_{\Lambda_F}^{\varphi,N_\nabla}\xleftarrow{\M_\Kis}\MF_{F\otimes_{\Qp}K}^{\varphi,N}\xrightarrow{\D_\Bre}\MF_{S_F}^{\varphi,N}.\]
We have seen that the functor $\D_\Bre$ gives us a natural way to construct a Breuil module from a filtered $\varphi$-module but this avenue is unsuitable to compute reductions. On the other hand, the functor $\D_{\Kis}$ is difficult to get a hold on structurally, but once done, we have a natural way to compute reductions. Our goal is then to bridge the gap between these two categories by first using $\D_\Bre$ to construct a Breuil module without monodromy before noticing that this can really be interpreted as a finite height Kisin module over some suitably large coefficient ring.

Define the functor
\[\D_{\BK}:\Mod_{S_F}^{\varphi,\le h}\rightarrow\MF_{S_F}^{\varphi,h}\]
by setting $\D_{\BK}(\Ms)\coloneqq S_F\otimes_{\varphi,S_F}\Ms$ such that $\varphi_{\D_\BK(\Ms)}=\varphi\otimes\varphi_{\Ms}$ and 
\[\Fil^h\D_\BK(\Ms)=\{x\in\D_\BK(\Ms):(1\otimes\varphi_\Ms)(x)\in\Fil^h S_F\cdot\Ms\}.\]
By \cite[Lemma 2.5]{BLL22}, $\Fil^h\D_{\BK}(\Ms)$ is finite free over $S_F$. Let us define the element $\gamma=\varphi(E)/p\in S_F^*$. The following theorem appears as \cite[Proposition 2.6]{BLL22}.

\begin{thm}\label{BK-equiv-thm}
	The functor $\D_{\BK}:\Mod_{S_F}^{\varphi,\le h}\rightarrow\MF_{S_F}^{\varphi,h}$ defines an equivalence of categories.
	\begin{proof}
		See the proof of \cite[Proposition 2.6]{BLL22} for full faithfulness. We replicate the proof of essential surjectivity here as it motivates Remark \ref{kis-alg-rem}.
		
		Suppose we are given a Breuil module $\Ds\in\MF_{S_F}^{\varphi,h}$ and choose bases $\{e_j\}_{j=1}^d$ of $\Ds$ and $\{\eta_j\}_{j=1}^d$ of $\Fil^h\Ds$. Under these bases we may write $(\eta_1,\dots,\eta_d)=(e_1,\dots,e_d)B$ and $\varphi_{\Ds}(e_1,\dots,e_d)=(e_1,\dots,e_d)X$ where $B\in\Mat_d(S_F)$ and $X\in\GL_d(S_F)$ due to the fact that the linearization of $\varphi_\Ds$ is necessarily invertible. Since $E^h\Ds\subseteq\Fil^h\Ds$ then there exists a matrix $A\in\Mat_d(S_F)$ such that $AB=BA=E^hI_d$. As a result, since $\varphi(E)=p\gamma\in S_F^*$ then $\varphi(B)\in\GL_d(S_F)$ so that 
		\[\varphi_\Ds(\eta_1,\dots,\eta_d)=(e_1,\dots,e_d)X\varphi(B)\]
		where $X\varphi(B)\in\GL_d(S_F)$.
		
		We introduce another basis $\{\mu_j\}_{j=1}^d$ defined by $(\mu_1,\dots,\mu_d)=(e_1,\dots,e_d)X\varphi(B)\varphi(E)^{-h}$ so that 
		$\varphi_\Ds(\eta_1,\dots,\eta_d)=(\mu_1,\dots,\mu_d)\varphi(E)^h$.	We necessarily then have that \[(\eta_1,\dots,\eta_d)=(\mu_1,\dots,\mu_d)(X\varphi(B)\varphi(E)^{-h})B\]
		and the finite height condition will imply the existence of $A'\in\Mat_d(S_F)$ such that the product $A'(X\varphi(B)\varphi(E)^{-h})B=E^hI_d$. Let us construct a finite free $S_F$-module $\Ms=\oplus_{j=1}^d S_F\rho_j$ such that $\rho_j\coloneqq 1\otimes \mu_j$ and set $\varphi_\Ms(\rho_1,\dots,\rho_d)=(\rho_1,\dots,\rho_d)A'$. The above discussion shows that $\Ms$ is a finite height Kisin module in $\Mod_{S_F}^{\varphi,\le h}$ and that $\D_{\BK}(\Ms)=\Ds$ as desired.
	\end{proof}
\end{thm}

To this point, we have seen that the functors $\D_\Kis$ and $\D_\Bre$ give us a way go to from filtered $\varphi$-modules to either Kisin or Breuil modules respectively. Additionally, the functor $\D_\BK$ allows us to switch between Kisin and Breuil modules over $S_F$. The following proposition does the job of connecting all of these functors together.

\begin{prop}\label{BK-iso-prop}
	Let $D\in\MF_{F\otimes_{\Qp}K}^{\varphi}$ be a filtered $\varphi$-module. Set $\Ds\coloneqq \D_\Bre(D)\in\MF_{S_F}^{\varphi,h}$ to be the image of $D$ under the forgetful functor forgetting monodromy and set $\Ms\coloneqq \M_\Kis(D)\otimes_{\Lambda_F}S_F\in\Mod_{S_F}^{\varphi,\le h}$. Then there is a natural isomorphism $\D_{\BK}(\Ms)\cong\Ds$.
	\begin{proof}
		See \cite[Thm 2.7]{BLL22}.
	\end{proof}
\end{prop}

\noindent We summarize the interaction of all these functors in the following diagram:
\begin{figure}[H]
	\centering
	\begin{tikzcd}
		& \MF_{F\otimes_{\Qp}K}^{\varphi} \arrow[ld, "\M_\Kis"'] \arrow[rd, "\D_\Bre"] &                                                    \\
		{\Mod_{\Lambda_F}^{\varphi,\le h}} \arrow[d, "\text{Extend}"']      &                                                                              & {\MF_{S_F}^{\varphi,N}} \arrow[d, "\text{Forget}"] \\
		{\Mod_{S_F}^{\varphi,\le h}} \arrow[rr, "\D_\BK"] &                                                                              & {\MF_{S_F}^{\varphi,h}}                           
	\end{tikzcd}
\end{figure}

\begin{rem}\label{kis-alg-rem}
	The proof of Theorem \ref{BK-equiv-thm} together with Proposition \ref{BK-iso-prop} allows us to describe an algorithm to explicitly compute a finite height Kisin module over $S_F$ from a filtered $\varphi$-module by passing first through the category of Breuil modules. In particular, we follow the steps:
	\begin{enumerate}[label=\arabic*)]
		\item Choose $S_F$-bases $\{e_1,\dots,e_d\}$ of $\Ds$ and $\{\alpha_1,\dots,\alpha_d\}$ of $\Fil^h\Ds$;
		\item Write $\varphi_\Ds(e_1,\dots,e_d)=(e_1,\dots,e_d)X$ and $(\alpha_1,\dots,\alpha_d)=(e_1,\dots,e_d)B$ for $X,B\in\Mat_d(S_F)$;
		\item Then, $\Ms$ admits an $S_F$-basis $\{\eta_1,\dots,\eta_d\}$ in which $\varphi_\Ms(\eta_1,\dots,\eta_d)=(\eta_1,\dots,\eta_d)A$ where
		\[A=E^hB^{-1}X\varphi(B)\varphi(E)^{-h}.\]
	\end{enumerate}
\end{rem}

\section{Strongly Divisible Lattices}\label{SDL-chap}

\subsection{\texorpdfstring{Lattices Inside Filtered {$\varphi$}-Modules}{Lattices Inside Filtered phi-Modules}}\label{lattice-inside-sec}

Let $D\in\MF_{F\otimes_{\Qp}K}^{\varphi}$ be a filtered $\varphi$-module over $F\otimes_{\Qp}K$. A \textit{$\varphi$-lattice} $L$ inside $D$ is an $\oh_F\otimes_{\Zp}\oh_K$-module $L$ with Frobenius action $\varphi_L$ such that
\[L\otimes_{\oh_F\otimes_{\Zp}\oh_K}(F\otimes_{\Qp}K)=D\hspace{.5cm}\text{and}\hspace{.5cm}\varphi_L(L)\subseteq L.\]
We say that $L$ has an \textit{induced filtration} from $D$ if $L$ admits a filtration structure such that
\[\Fil^j L=\Fil^j D\cap L.\]
From now on, we will always assume that any $\varphi$-lattice $L$ inside a filtered $\varphi$-module $D$ has an induced filtration. 

Recall that by Proposition \ref{fil-idem-decomp-prop}, we may decompose $D$ along the $f$ embeddings $\tau_i:K\hookrightarrow F$ so that $D=\prod_i D^{(i)}$. We would like to descend this isomorphism to the integral case which the following lemma from \cite[Lem 3.1]{Zhu09} achieves.

\begin{lem}\label{int-tensor-decomp-lem}
	The isomorphism $\theta:F\otimes_{\Q_p}K\rightarrow F^{|\tau|}$ of Lemma \ref{tensor-decomp-lem} induces an isomorphism of $\oh_F\otimes_{\Z_p}\oh_K$-modules, $\theta:\oh_F\otimes_{\Z_p}\oh_K\rightarrow\prod_{i\in\Z/f\Z}\oh_F\eqcolon \oh_F^{|\tau|}$.
	\begin{proof}
		The proof of the lemma is completely similar to that of Lemma \ref{tensor-decomp-lem} once we show that $\theta$ is an isomorphism of integral rings. $\oh_K$ is generated by a single element over $\Zp$ so there exists an irreducible polynomial $g\in\Zp[x]$ of degree $f$ such that $\oh_K\cong \Zp[x]/(g)$. Hence,
		\[\oh_F\otimes_{\Zp}\oh_K\cong\oh_F\otimes_{\Z_p}\Zp[x]/(g)\cong \oh_F[x]/(g).\]
		Since $K\subseteq F$ then $\oh_K\subseteq \oh_F$ and any roots of $g$ that lie in $\oh_K$ will lie in $\oh_F$. Thus $\oh_F[x]/(g)\cong\prod_{i\in\Z/f\Z} \oh_F$.
	\end{proof}
\end{lem}

Let $L$ be a $\varphi$-lattice inside $D$ with induced filtration. The idempotents $\idem_i$ inside $F^{|\tau|}$ can be viewed as lying inside $\oh_F^{|\tau|}$. By definition of the induced filtration, we can write $L^{(i)}\coloneqq\idem_i L=D^{(i)}\cap L$. The map $\theta:L\rightarrow\prod_i L^{(i)}$ is an isomorphism of $\oh_F\otimes_{\Z_p}\oh_K$-modules by Lemma \ref{int-tensor-decomp-lem} so that we may write $\varphi_L=(\varphi_{L}^{(0)},\dots,\varphi_{L}^{(f-1)})$ with $i$-th partial Frobenius action $\varphi_L^{(i)}:\varphi^*L^{(i-1)}\rightarrow L^{(i)}$.

\begin{prop}\label{phi-lat-decomp}
	There is an isomorphism of $\varphi$-lattices with induced filtration
	\[L\cong\prod_{i\in\Z/f\Z}L^{(i)}.\]
	\begin{proof}
		$\theta$ is compatible with the filtration and Frobenius structures on $D^{(i)}$ by Proposition \ref{fil-idem-decomp-prop}. The fact that $L^{(i)}=D^{(i)}\cap L$ will imply the same is true on $L^{(i)}$.
	\end{proof}
\end{prop}

\noindent Just as for filtered $\varphi$-modules, the results of Proposition \ref{phi-lat-decomp} allows us to pass between a $\varphi$-module $L$ and its decomposition $\prod_i L^{(i)}$ which we do so freely. Each $L^{(i)}$ may then be viewed as $\oh_F$-modules via the isomorphism between $\idem_i\oh_F$ and $\oh_F$.

An $\oh_F$-basis $\{e_1^{(i)},\dots,e_d^{(i)}\}$ of $L^{(i)}$ is said to be \textit{adapted to the induced filtration} if the set $\{e_1^{(i)},\dots,e_{\dim\Fil^jL^{(i)}}^{(i)}\}$ forms a basis for $\Fil^j L^{(i)}$ for all $j$. A filtered $\varphi$-module $D=\prod_i D^{(i)}$ is said to have a basis adapted to the filtration if it contains a $\varphi$-lattice $L=\prod_i L^{(i)}\subset D$ which has a basis adapted to the induced filtration for all $i\in\Z/f\Z$. The existence of such bases relies on the following proposition using ideas from \cite[Prop 2.1]{Zhu09}.

\begin{prop}\label{exist-adapt-basis}
	Let $D=\prod_i D^{(i)}$ be a $d$-dimensional filtered $\varphi$-module over $F\otimes_{\Qp}K$ and let $L=\prod_i L^{(i)}\subset D$ be a $\varphi$-lattice contained in $D$. Then there exists a basis of $L$ adapted to the filtration induced from $D$.
	\begin{proof}
		Without loss of generality, let us work with a fixed $i\in\Z/f\Z$. We first observe that each $\Fil^{j+1}L^{(i)}$ is a saturated submodule of $\Fil^j L^{(i)}$. Indeed, if we let $r\in\oh_F$ and $x\in\Fil^j L^{(i)}$ be such that $rx\in\Fil^{j+1}L^{(i)}$, then $x\in\Fil^{j+1}D^{(i)}$ since $\Fil^{j+1}L^{(i)}$ spans $\Fil^{j+1}D^{(i)}$ over $F$. Since the filtration of $L$ is induced from $D$ then $x\in\Fil^{j+1}L^{(i)}$.
		
		Take $n$ to be the largest integer such that $\Fil^n L^{(i)}$ is nontrivial and choose a basis $\{e_1^{(i)},\dots,e_{r_n}^{(i)}\}$ of $\Fil^n L^{(i)}$ where $r_n=\Rank(\Fil^n L^{(i)})$. Since $\Fil^{n}L^{(i)}$ is saturated in $\Fil^{n-1}L^{(i)}$ then $e_i^{(i)}\in\Fil^{n-1}L^{(i)}$ and there exists a basis of $\Fil^{n-1}L^{(i)}$ such that the quotient can be written
		\[\Fil^{n-1}L^{(i)}/\Fil^n L^{(i)}=\langle \overline{e}_{r_n+1}^{(i)},\dots, \overline{e}_{r_{n-1}}^{(i)}\rangle\]
		where $r_{n-1}=\Rank(\Fil^{n-1}L^{(i)})$. By choosing a lift $e_j^{(i)}\in \Fil^{n-1}L^{(i)}$ of $\overline{e}_j^{(i)}$ then $\Fil^{n-1}L^{(i)}$ has basis \[\{e_1^{(i)},\dots,e_{r_n}^{(i)},e_{r_n+1}^{(i)},\dots,e_{r_{n-1}}^{(i)}\}.\] 
		Completing an induction up to when $r_i=d$ will give a basis $\{e_1^{(i)},\dots,e_d^{(i)}\}$ of $L^{(i)}$ that is adapted to the induced filtration.
	\end{proof}
\end{prop}

As a result of Proposition \ref{exist-adapt-basis}, we will take a $\varphi$-lattice $L$ to always be generated by a basis that is adapted to the induced filtration. This means the filtration structure is completely described by the Hodge polygon data $\kvec$ and so the Frobenius action determines isomorphism classes of such modules. This idea is expanded upon in Section \ref{class-SDL-sec}.

The weakly admissible condition on a filtered $\varphi$-module $D$ implies some compatibility between the Frobenius and filtration structures on $D$. As a result, a $\varphi$-lattice $L$ inside $D$ should inherit some compatibility between the two induced structures whenever $D$ is weakly admissible. The next definition gives the required condition on such $\varphi$-lattices. 

\begin{defin}\label{strong-div-lattice-def}
	A $\varphi$-lattice $L\subset D$ is said to be \textit{strongly divisible} if it satisfies the further properties
	\[\varphi_L(\Fil^jL)\subseteq p^j L\hspace{.5cm}\text{and}\hspace{.5cm}\sum_{j\in\Z}\frac{1}{p^j}\varphi_L(\Fil^jL)=L.\]
\end{defin}

\noindent We define the category of strongly divisible $\varphi$-lattices over $\oh_F\otimes_{\Z_p}\oh_K$ by $\SD_{\oh_F\otimes_{\Zp}\oh_K}^\varphi$. We obtain a natural functor $\SD_{\oh_F\otimes_{\Zp}\oh_K}^{\varphi}\rightarrow\MF_{F\otimes_{\Qp}K}^{\varphi}$ given by scalar extension up to $F\otimes_{\Qp}K$.

We would like to translate this strongly divisible condition to be in terms of the decomposition $L=\prod_i L^{(i)}$. This is done by the following lemma from \cite[Prop 3.2]{Zhu09}.

\begin{lem}\label{strong-div-idem-decomp-lem}
		A $\varphi$-lattice $L\subset D$ is strongly divisible if and only if  when viewing $L=\prod_{i}L^{(i)}$, we have both
		\[\varphi_L^{(i)}(\Fil^j L^{(i-1)})\subseteq p^{j}L^{(i)}\hspace{.5cm}\text{and}\hspace{.5cm}\sum_{j\in\Z}p^{-j}\varphi_L^{(i)}(\Fil^j L^{(i-1)})=L^{(i)}\]
		for all $i\in\Z/f\Z$.
	\begin{proof}
		We know that $L$ is strongly divisible if and only if 
		\[\varphi_L(\Fil^jL)\subseteq p^j L\hspace{.5cm}\text{and}\hspace{.5cm}\sum_{i\in\Z}\frac{1}{p^j}\varphi_L(\Fil^jL)=L.\]
		Since the idempotent decomposition $L=\prod_{i\in\Z/f\Z}L^{(i)}$ respects both the filtration and Frobenius structures, we may write
		\[\prod_{i\in\Z/f\Z}\varphi_{L}^{(i)}(\Fil^j L^{(i-1)})=\varphi_L(\Fil^jL)\subseteq p^jL=\prod_{i\in\Z/f\Z}p^jL^{(i)}\]
		allowing us to conclude that $\varphi_L^{(i)}(\Fil^j L^{(i-1)})\subseteq p^{j}L^{(i)}$. Furthermore, we may also write
		\[\sum_{j\in\Z}\prod_{i\in\Z/f\Z}\frac{1}{p^j}\varphi_{L}^{(i)}(\Fil^jL^{(i-1)})=\sum_{i\in\Z}\frac{1}{p^j}\varphi_L(\Fil^jL)=L=\prod_{i\in\Z/f\Z}L^{(i)}.\]
		Orthogonality of the idempotent decomposition will then allow us to deduce that the equality $\sum_{j\in\Z}p^{-j}\varphi_L^{(i)}(\Fil^j L^{(i-1)})=L^{(i)}$. The only if direction follows by following the above argument in reverse.
	\end{proof}
\end{lem}

In \cite{Laf80}, Laffaille showed that the weak admissibility condition on a filtered $\varphi$-module $D$ over $K$ is equivalent to the existence of a strongly divisible lattice over $\oh_K$ inside $D$. The next proposition from \cite[Prop 3.3]{Zhu09} may be viewed as an extension of this result to $F$-valued scalars.

\begin{prop}\label{weak-admis-equiv-SD-prop}
	Let $D$ be a $d$-dimensional filtered $\varphi$-module over $F\otimes_{\Q_p}K$ with Hodge polygon data $\HP(D)=\kvec$. The following statements are equivalent:
	\begin{enumerate}
		\item $D$ is weakly admissible;
		\item There exists a strongly divisible lattice $L$ over $\oh_F\otimes_{\Z_p}\oh_K$ inside $D$;
		\item There exists a $\varphi$-lattice $L=\prod_i L^{(i)}$ over $\oh_F\otimes_{\Z_p}\oh_K$ inside $D$ such that with respect to a basis adapted to the induced filtration, $\varphi_{L}^{(i)}=A^{(i)}\Delta_{\kvec^{(i-1)}}$ where $A^{(i)}\in\GL_2(\oh_F)$ and $\Delta_{\kvec^{(i)}}=\Diag(p^{k_{i1}},\dots,p^{k_{id}})$ for all $i\in\Z/f\Z$.
	\end{enumerate}
	\begin{proof}
		The fact that $(b)\Rightarrow (a)$ is due to \cite[\S 3]{Laf80} and then extending scalars to $\oh_F$. To see that $(a)\Rightarrow (b)$, then by \cite{Laf80} again, weak admissibility implies the existence of a strongly divisible lattice $L'$ over $\oh_K$ since weak admissibility is a condition on $K$. By setting $L=L'\otimes_{\Z_p}\oh_F\coloneqq \sum_{j=1}^{n}L'e_i$ where $\oh_F=\langle e_1,\dots,e_n\rangle$ over $\Z_p$, then we have an $\oh_K\otimes_{\Z_p}\oh_E$ module and since filtrations and Frobenius are $F$-linear,
		\[\sum_{j\in\Z}\frac{1}{p^j}\varphi(\Fil^jL)=\sum_{\ell=0}^{n}\sum_{j\in\Z}\frac{1}{p^j}\varphi(\Fil^jL')e_\ell=\sum_{\ell=0}^{n}L'e_\ell=L.\]
		Thus $L$ is strongly divisible over $\oh_F\otimes_{\Z_p}\oh_K$ and we have that $(a)\Leftrightarrow (b)$.
		
		For $(b)\Rightarrow (c)$, Lemma \ref{strong-div-idem-decomp-lem} says that we must have $\sum_{j\in\Z}p^{-j}\varphi_L^{(i)}(\Fil^j L^{(i-1)})=L^{(i)}$ and $\varphi_L^{(i)}(\Fil^j L^{(i-1)})\subseteq p^{j}L^{(i)}$. By Lemma \ref{exist-adapt-basis}, we may choose a basis $\{e_1^{(i)},\dots,e_d^{(i)}\}$ of $L^{(i)}$ which is adapted to the induced filtration. For $1\le j\le d$, set $r_{ij}=\dim\Fil^{k_{ij}}L^{(i)}$. Then the fact that $\varphi_L^{(i)}(\Fil^j L^{(i-1)})\subseteq p^{j}L^{(i)}$ implies for each $1\le j\le d$, 
		
		\[\varphi_L^{(i)}(e_1^{(i-1)},\dots,e_{r_{ij}}^{(i-1)})=(e_1^{(i)},\dots,e_d^{(i)})A^{(i)}_{r_{(i-1)j}}\begin{pmatrix}
			\Diag(p^{k_{(i-1)1}},\dots, p^{k_{{(i-1)j}}}) \\ 0 
		\end{pmatrix}\]
		for some $A^{(i)}_{r_{(i-1)j}}\in\Mat_{r_{(i-1)j}}(\oh_F)$. Inducting on $j$ will give us $\Mat(\varphi_L^{(i)})=A^{(i)}\Delta_{\kvec^{(i-1)}}$. The fact that $\sum_{j\in\Z}p^{-j}\varphi_L^{(i)}(\Fil^j L^{(i-1)})=L^{(i)}$ will imply that the columns of $A^{(i)}$ will generate $L^{(i)}$ over $\oh_F$ so that $A^{(i)}\in\GL_d(\oh_F)$. The $(c)\Rightarrow(b)$ direction follows from Lemma \ref{strong-div-idem-decomp-lem}.		
	\end{proof}
\end{prop}

\noindent As a result, we may refine our scalar extension functor $\SD_{\oh_F\otimes_{\Zp}\oh_K}^{\varphi}\rightarrow\MF_{F\otimes_{\Qp}K}^{\varphi,w.a.}$. Restricting to strongly divisible lattices arising out of weakly admissible filtered $\varphi$-modules of prescribed Hodge newton polygon data $\kvec$, we obtain the proper subcategory $\SD_{\oh_F\otimes_{\Zp}\oh_K}^{\varphi,\kvec}$.

In practice, the equivalence of Proposition \ref{weak-admis-equiv-SD-prop} give us a concrete way to describe the structure of crystalline representations. Since isomorphism classes of crystalline representations are directly related to isomorphism classes of filtered $\varphi$-modules by the Colmez-Fontaine equivalence of Theorem \ref{CF-equiv-w-coeff-thm}, then in order to determine such isomorphism classes, we need only classify the strongly divisible lattices associated to those filtered $\varphi$-modules. The following section gives us the machinery to do just this.

\subsection{Classifying Strongly Divisible Lattices}\label{class-SDL-sec}
Let $D$ continue to be a weakly admissible filtered $\varphi$-module over $F\otimes_{\Qp}K$ with Hodge polygon data $\HP(D)=\kvec$. Also let $L$ be a strongly divisible $\varphi$-lattice inside $D$ and assume that it has a basis $\{e_1,e_2,\dots,e_d\}$ which is adapted to the induced filtration which is guaranteed to exist by Proposition \ref{weak-admis-equiv-SD-prop}. In this case, for each labeled Hodge-Tate weight $k_{ij}\in\kvec^{(i)}\subset\kvec$, we define $r_{ij}\coloneqq\dim_F\Fil^{k_{ij}}D^{(i)}$.

\begin{defin}\label{para-grp-def}
	We define the \textit{parabolic group} of $\kvec^{(i)}$ over $\oh_F$ to be the proper subgroup $\Para_{\kvec^{(i)}}(\oh_F)\subset\GL_d(\oh_F)$ defined by
	\[\Para_{\kvec^{(i)}}(\oh_F)\coloneqq\begin{pmatrix}
		\GL_{r_{i1}}(\oh_F) & \star & \cdots & \star \\
		0 & \GL_{r_{i2}-r_{i1}}(\oh_F) & \star & \vdots \\
		\vdots & & \ddots & \star\\
		0 & \dots & 0 & \GL_{d-r_{i(d-1)}}(\oh_F)
	\end{pmatrix}\hspace{1cm}\]
	where $\star$ denotes arbitrary elements in $\oh_F$.
	
	We can then define the \textit{parabolic group} of $\kvec$ over $\oh_F$ to be the subgroup $\Para_\kvec(\oh_F)\subset\GL_d(\oh_F)^f$ where $\Para_\kvec\coloneqq(\Para_{\kvec^{(0)}}(\oh_F),\dots,\Para_{\kvec^{(f-1)}}(\oh_F)) $. 
\end{defin}


We can use our parabolic groups to induce an equivalence relation on matrices. Recall that for any $\kvec^{(i)}=(k_{i1},\dots,k_{id})\subset \kvec$, we define $\Delta_{\kvec^{(i)}}=\Diag(p^{k_{i1}},\dots,p^{k_{id}})$. Let $(A)=(A^{(0)},\dots A^{(f-1)})$ and $(B)=(B^{(0)},\dots,B^{(f-1)})$ be two $f$-tuples in $\GL_d(\oh_F)^f$. We say that $(A)$ and $(B)$ are \textit{parabolic equivalent with respect to $\kvec$}, written $(A)\sim_{\kvec}(B)$, if there exists $(C)=(C^{(0)},\dots,C^{(f-1)})\in\Para_\kvec(\oh_F)$ such that for each $0\le i\le f-1$, we have
\[B^{(i)}=C^{(i)}A^{(i)}\Delta_{\kvec^{(i-1)}}\left(C^{(i-1)}\right)^{-1}\Delta_{\kvec^{(i-1)}}^{-1}.\]
In this case, we will commonly write $(C)*_\Para(A)=(B)$. The next theorem inspired from \cite[Thm 3.5]{Zhu09} shows that we can use parabolic equivalence classes to define isomorphism classes of strongly divisible lattices. 

\begin{thm}\label{para-equiv-class-thm}
	There is a bijection
	\[\Theta:\GL_d(\oh_F)^f/\sim_{\kvec}\longrightarrow\SD_{\oh_F\otimes_{\Zp}\oh_K}^{\varphi,\kvec}\]
	defined by sending an $f$-tuple $(A)=(A^{(0)},\dots,A^{(f-1)})\in\GL_d(\oh_F)^f$ to a strongly divisible lattice $L=\prod_iL^{(i)}$ inside $D\in\MF_{F\otimes_{\Qp}K}^{\varphi,w.a.,\kvec}$ in a basis adapted to the filtration where $\varphi_{L}^{(i)}=A^{(i)}\Delta_{\kvec^{(i-1)}}$.
	\begin{proof}
		We begin with the reverse direction. Let $L,L'\in\SD_{\oh_F\otimes_{\Z_p}\oh_K}^{\varphi,\kvec}$ be isomorphic strongly divisible lattices inside $D\in\MF_{F\otimes_{\Qp}K}^{\varphi,w.a.,\kvec}$ via an isomorphism $h:L\rightarrow L'$. Proposition \ref{weak-admis-equiv-SD-prop} says that we may write $L=\prod_i L^{(i)}$ and $L'=\prod_i L'^{(i)}$ with Frobenius actions given by $\varphi_L^{(i)}=A^{(i)}\Delta_{\kvec^{(i-1)}}$ and $\varphi_{L'}^{(i)}=A'^{(i)}\Delta_{\kvec^{(i-1)}}$ respectively. Restricting our filtration and $\varphi$-compatible isomorphism to each embedding piece gives $h^{(i)}(L^{(i)})\cong L'^{(i)}$, $h^{(i)}(\Fil^j L^{(i)})\cong \Fil^j L'^{(i)}$ and $h^{(i)}\circ\varphi_L^{(i)}=\varphi_{L'}^{(i)}\circ h^{(i-1)}$. These conditions on $h^{(i)}$ implies that for $1\le j\le d$ if we set $r_{ij}=\dim \Fil^jL^{(i)}$, we have
		\[h^{(i)}(e_{r_{i(j-1)}}^{(i)},\dots, e_{r_{ij}}^{(i)})=(e'^{(i)}_{r_{i(j-1)}},\dots, e'^{(i)}_{r_{ij}})C^{(i)}_{r_{ij}}\]
		where each $C^{(i)}_{r_{ij}}\in\GL_{r_{ij}-r_{i(j-1)}}(\oh_F)$. Hence, we have $h^{(i)}(e_1^{(i)},\dots,e_d^{(i)})=(e_1^{(i)},\dots,e_d^{(i)})C^{(i)}$ where
		\[C^{(i)}\coloneqq\begin{pmatrix}
			C_{r_{i1}}^{(i)} & \star & \cdots & \star \\
			0 & C_{r_{i2}}^{(i)} & & \vdots \\
			\vdots & & \ddots & \star\\
			0 & \cdots & 0 & C_{r_{id}}^{(i)}
		\end{pmatrix}\in \Para_{\kvec^{(i)}}(\oh_F)\]
		which gives $\Mat(h)\in\Para_\kvec(\oh_F)$. We then have the matrix relation from $\varphi$-compatibility
		\[A^{(i)}\Delta_{\kvec^{(i-1)}}=C^{(i)}A'^{(i)}\Delta_{\kvec^{(i-1)}}(C^{(i-1)})^{-1}\]
		so we conclude that $(C) *_\Para (A) = (A')$.
		
		The other direction follows from performing the above argument in reverse.
	\end{proof}
\end{thm}

As a consequence, we will commonly write $\Theta(A)=L(A)=\prod_i L(A)^{(i)}$ to denote the associated strongly divisible lattice with $i$-th partial Frobenius action given by $A^{(i)}\Delta_{\kvec^{(i-1)}}$. We define
\[D(A)\coloneq L(A)\otimes_{\oh_F\otimes_{\Zp}\oh_K}F\otimes_{\Qp}K\]
to be the associated weakly admissible filtered $\varphi$-module with Hodge Polygon data $\kvec$ by Proposition \ref{weak-admis-equiv-SD-prop}. By the aforementioned proposition, all such weakly admissible filtered $\varphi$-modules arise out of this construction. Recalling the equivalence functor $\V_\cris$ of Theorem \ref{CF-equiv-w-coeff-thm}, there exists a crystalline representation with Hodge-Tate weight data $\kvec$ given by
\[V(A)\coloneq \V_\cris^*(D(A))\]
which likewise exhausts all such representations. This means we can use parabolic equivalence classes to pick out `simple' representatives for the isomorphism classes of crystalline representations. More on this in Section \ref{models-irred-reps-chap}.

We conclude this section with looking that how the bijection of Theorem \ref{para-equiv-class-thm} behaves with respect to the notion of taking tensor products. This will be important as it will allow us to twist by crystalline characters in order to adjust the labeled Hodge-Tate weights of our representations for more convenient computations. Consider strongly divisible lattices  $\Theta(A)=L(A)$ and $\Theta(A')=L(A')$ of $\oh_F\otimes_\Zp\oh_K$ dimension $d$ and $d'$ respectively and arising out of weakly admissible filtered $\varphi$-modules with Hodge Polygon data $\kvec$ and $\kvec'$ respectively.  

We define two tensor products
\begin{itemize}
	\item $\kvec\otimes\kvec'\coloneqq \left((k_{ij}+k'_{ij'})_{j,j'}\right)_{i\in\Z/f\Z}$ where $1\le i\le d$ and $1\le i'\le d'$ reordering if necessary.
	\item $A\otimes A'\coloneqq (A^{(i)}\otimes A'^{(i)})_{i\in\Z/f\Z}\in \GL_{dd'}(\oh_F)^f$ where $A^{(i)}\otimes A'^{(i)}$ denotes the normal Kronecker tensor product.
\end{itemize}

\noindent The result we desire is given in the following proposition from \cite[Prop 3.7]{Zhu09}.

\begin{prop}\label{bij-tensor-prop}
	Let notation be as above so that $\Theta(A)=L(A)\in\SD_{\oh_F\otimes_{\Zp}\oh_K}^{\varphi,\kvec}$ and $\Theta(A')=L(A')\in\SD_{\oh_F\otimes_{\Zp}\oh_K}^{\varphi,\kvec'}$. Then $\Theta(A\otimes A')=L(A)\otimes L(A')\in\SD_{\oh_F\otimes_{\Zp}\oh_K}^{\varphi,\kvec\otimes\kvec'}$.
	\begin{proof}
		First, note that 
		\[L(A)\otimes L(A')=\prod_i L(A)^{(i)}\otimes \prod_i L(A')^{(i)}=\prod_i L(A)^{(i)}\otimes L(A')^{(i)}\]
		with the last equality following from the orthogonality of the idempotent decomposition. Both $L(A)$ and $L(A')$ are strongly divisible so Proposition \ref{weak-admis-equiv-SD-prop} implies the existence of a basis adapted to the filtration so that the $i$-th partial Frobenius action of $L(A)\otimes L(A')$ is given by
		\[ (A^{(i)}\Delta_{\kvec^{(i-1)}})\otimes(A'^{(i)}\Delta_{\kvec'^{(i-1)}})= (A^{(i)}\otimes A'^{(i)})(\Delta_{\kvec^{(i-1)}}\otimes\Delta_{\kvec'^{(i-1)}})\]
		where the equality follows form the mixed product property of the Kronecker tensor product. After observing the fact that 
		\[\Delta_{\kvec^{(i-1)}}\otimes\Delta_{\kvec'^{(i-1)}}=\Diag(p^{k_{(i-1)1}+k'_{(i-1)1}},\dots,p^{k_{(i-1)d}+k'_{(i-1)d'}})\]
		then it is clear that $\Theta(A\otimes A')=L(A)\otimes L(A')\in\SD_{\oh_F\otimes_{\Zp}\oh_K}^{\varphi,\kvec\otimes\kvec'}$.
	\end{proof}
\end{prop}

\noindent The results of Proposition \ref{bij-tensor-prop} give us a way to simplify the Hodge-Tate weight data $\kvec$ of a crystalline representation $V$. 

\begin{rem}\label{weight-twist-rem}
	Consider the representation $V(A)\in\Rep_{\cris/F}^\kvec(G_K)$ of dimension $d$ and define a crystalline character $V(I_d)\in \Rep_{\cris/F}^{\kvec'}(G_K)$ where $\kvec'=(-k_{id})_i$. Then the above results will imply that $V(A)\otimes V(I_d)$ will have Hodge-Tate weight data given by $(k_{i1}-k_{id},\dots,0)_i$ but leave the matrices $(A)$ unaffected. Hence, up to a twist by a crystalline character, we may always assume that $k_{id}=0$ for all $i\in\Z/f\Z$.
\end{rem}

\section{Isomorphism Classes in Dimension Two}\label{models-irred-reps-chap}

\subsection{Parabolic Equivalence Classes}\label{para-equiv-class-sec}

From Section \ref{class-SDL-sec}, recall that we may induce an equivalence relation $\sim_{\kvec}$ on $\GL_2(\oh_F)^f$  by understanding that $(A^{(i)})\sim_{\kvec}(B^{(i)})$ whenever there exists $(C^{(i)})\in\Para_\kvec(\oh_F)$ such that
\[B^{(i)}=C^{(i)}A^{(i)}\Delta_{\kvec^{(i-1)}}\left(C^{(i-1)}\right)^{-1}\Delta_{\kvec^{(i-1)}}^{-1}\]
for all $i\in\Z/f\Z$. Under the assumption that dimension $d=2$ and the labeled Hodge-Tate weights are \textit{regular} in the sense that $k_i>0$ for all $i\in\Z/f\Z$, we may realize that each parabolic group $\Para_{\kvec^{(i)}}(\oh_F)$ is nothing more than the invertible upper triangular matrices over $\oh_F$ and $\Delta_{\kvec^{(i)}}=\Diag(p^{k_i},1)$ by Remark \ref{weight-twist-rem}. The purpose of this section is to pick equivalence class representatives that are as simple as possible by choosing suitable parabolic matrices $(C^{(i)})$.

Let $[(A^{(i)})]\in\GL_2(\oh_F)^f/\sim_{\kvec}$ define an arbitrary equivalence class where each
\[A^{(i)}=\begin{pmatrix}a^{(i)}_{11} & a^{(i)}_{12} \\ a^{(i)}_{21} & a^{(i)}_{22}\end{pmatrix}.\]
Since each matrix is invertible, then we can safely assume that at least one of $a^{(i)}_{21}$ and $a^{(i)}_{22}$ is invertible. Following the ideas found in \cite[\S~2]{Liu21}, we will explicitly construct parabolic matrices $(C^{(i)})\in\Para_\kvec(\oh_F)$ by
\[C^{(i)}=\begin{pmatrix}c^{(i)}_{11} & c^{(i)}_{12} \\ 0 & c^{(i)}_{22}\end{pmatrix}\]
such that the result of $(C^{(i)})*_\Para(A^{(i)})$ has $\begin{psmallmatrix}0 \\ \beta\end{psmallmatrix}$ as one of its columns where $\beta\in\oh_F^*$. In order to ease notation, let us begin to write $(xy)^{(i)}=x^{(i)}y^{(i)}$.

\begin{lem}\label{para-equiv-inv-lem} Suppose $k_i>0$ for all $i\in\Z/f\Z$ and $(A^{(i)})\sim_{\kvec}(B^{(i)})$. Then $a^{(i)}_{2\ell}\in\oh_F^*$ if and only if $b^{(i)}_{2\ell}\in\oh_F^*$.
	\begin{proof}
		Let $(C^{(i)})$ be the parabolic equivalent change of basis turning $(A^{(i)})$ into $(B^{(i)})$ using notation as above. We have
		\[b_{21}^{(i)}=\frac{(a_{21} c_{11})^{(i)}}{c_{22}^{(i-1)}}\hspace{2cm}b_{22}^{(i)}=\frac{(a_{22} c_{22})^{(i)}+(a_{21} c_{12})^{(i)} p^{k_i}}{c_{22}^{(i-1)}}.\]
		The lemma follows from the fact that $c_{11}^{(i)},c_{22}^{(i)}\in\oh_F^*$ for all $i$.
	\end{proof}
\end{lem}

\begin{lem}\label{para-equiv-killing-lem}
	For all $i\in\Z/f\Z$, suppose $k_i>0$ and $a^{(i)}_{2\ell}\in\oh_F^*$ for some $\ell\in\{1,2\}$. Then there exists a parabolic equivalent change of basis $(C^{(i)})\in\Para_\kvec(\oh_F)$ such that by setting $(C^{(i)})*_\Para(A^{(i)})=(B^{(i)})$, then $b_{1\ell}^{(i)}=0$.
	\begin{proof}
		We must treat two cases separately. First suppose that there exists some $r\in\Z/f\Z$ such that $A^{(r)}$ has $a^{(r)}_{21}\in\oh_F^*$. For any fixed $i\in\Z/f\Z$, observe that if $a_{2\ell}^{(i)}\in\oh_F^*$ then by taking 
		\begin{equation}\label{para-change-equ}
		C^{(i)}=\begin{pmatrix}
		1 & -\frac{a_{1\ell}^{(i)}}{a^{(i)}_{2\ell}} \\ 0 & \frac{1}{a^{(i)}_{2\ell}}
		\end{pmatrix}\in\Para_{\kvec^{(i)}}(\oh_F)\hspace{.25cm}\text{ and }\hspace{.25cm} C^{(j)}=I\hspace{.25cm}\text{ for }\hspace{.25cm}j\neq i
		\end{equation}
		will give $B^{(i)}$ with $b^{(i)}_{1\ell}=0$ and $b^{(i)}_{2\ell}=1$. We notice that this change of basis also affects the right column of $B^{(i+1)}$ (but not the invertibility of $b^{(i+1)}_{22}$ by Lemma \ref{para-equiv-inv-lem})and does not affect $B^{(j)}$ for $j\notin \{i,i+1\}$. Since this base change is in $\Para_\kvec(\oh_F)$ and we are only concerned with matrices up to parabolic equivalence, then we can safely reset $(A^{(i)})$ to be $(B^{(i)})$ and repeat this process of \ref{para-change-equ} for $i+1$ and so on until we return to $i$ taking the indexes modulo $f$. If we start this process at the aforementioned $A^{(r)}$ then when we perform the final step to base change $A^{(r-1)}$ then the left hand column of $A^{(r)}$ is unaffected.  The result is a parabolic equivalent $f$-tuple $(B^{(i)})$ with one column which looks like $\begin{psmallmatrix}
			0 \\ 1
		\end{psmallmatrix}$.
	
		The other case is when there is no such $r$ with $a^{(r)}_{21}\in\oh_F^*$; Hence, we want to make $a^{(i)}_{12}=0$ for all $i\in\Z/f\Z$. Here, the above will not work as the final step of the process will make the upper right hand entry of the initial matrix non-zero. Instead, let us take
		\[C^{(i)}=\begin{pmatrix}
			1 & -\frac{a_{12}^{(i)}}{a^{(i)}_{22}} \\ 0 & \frac{1}{a^{(i)}_{22}}
		\end{pmatrix}\in\Para_{\kvec^{(i)}}(\oh_F)\hspace{.25cm}\text{ for all }\hspace{.25cm}i\in\Z/f\Z.\]
		The result will leave $(B^{(i)})$ with $b^{(i)}_{22}\in\oh_F^*$ by Lemma \ref{para-equiv-inv-lem} and $\nu_p(b^{(i)}_{12})\ge k_i$. Since $k_i>0$ for all $i\in\Z/f\Z$ then by resetting $(A^{(i)})$ to be $(B^{(i)})$ and repeating the process, we will have that $\nu_p(b^{(i)}_{12})\rightarrow\infty$ and $b^{(i)}_{22}\in\oh_F^*$.
	\end{proof}
\end{lem}

\noindent The results of the above lemmas motivates the following definition.

\begin{defin}\label{matrix-Type-def}
	We define a \textit{Type} matrix to be:
	\begin{itemize}
		\item Type $\I$ matrix: $\begin{pmatrix}0 & m_{1} \\ 1 & m_{2} \end{pmatrix}$ where $m_{1}\in\oh_F^*$ and $m_{2}\in\oh_F$;
		\item Type $\II_\alpha$ matrix: $\begin{pmatrix}m_{1} & 0 \\ m_{2} & \alpha\end{pmatrix}$ where $m_{1},\alpha\in\oh_F^*$ and $m_{2}\in\varpi\oh_F$.
	\end{itemize}
	We define the associated Type map to be $\Ts:\GL_2(\oh_F)\rightarrow\{\I,\II\}$ where for any $A\in\GL_2(\oh_F)$, $\Ts(A)=\I$ if and only if $a_{21}\in\oh_F^*$ and $\Ts(A)=\II$ otherwise. 
\end{defin}

In other words, the above definition says that for an $f$-tuple $(A^{(i)})\in\GL_2(\oh_F)^f$, for each $i$, $\Ts(A^{(i)})=\I$ or $\II$ whenever $(A^{(i)})\sim_{\kvec}(B^{(i)})$ with $B^{(i)}$ being a Type $\I$ or $\II_\alpha$ matrix respectively. We will refer to the $f$-tuple $(\Ts(A^{(0)}),\dots,\Ts(A^{(f-1)}))$ as being the \textit{Type combination} of $(A)$. Since parabolic equivalence does not effect the invertibility of the bottom row by Lemma \ref{para-equiv-inv-lem}, then the Type map is well-defined.

\begin{prop}\label{class-para-equiv-classes-prop}
	The parabolic equivalence classes of $\GL_2(\oh_F)^f/\sim_{\kvec}$ for $k_i>0$ may be represented by an $f$-tuple of matrices $(A^{(i)})$ satisfying one of two cases:
	\begin{enumerate}
		\item For each $i$, the matrix $A^{(i)}$ is a Type $\I$ or $\II_1$ matrix;
		\item For all $i$, the matrix $A^{(i)}$ is a Type $\II_{\alpha^{(i)}}$ matrix for some $\alpha^{(i)}\in\oh_F^*$.
	\end{enumerate}
	\begin{proof}
		For a general equivalence class $[(B^{(i)})]\in\GL_2(\oh_F)^f/\sim_{\kvec}$, Lemma \ref{para-equiv-inv-lem} ensures that elements of the bottom row of every member of this class has the same invertiblity statuses. The proposition follows from the proof of Lemma \ref{para-equiv-killing-lem}. Namely, if for at least one $B^{(i)}$ we have $b^{(i)}_{21}\in\oh_F^*$ then we are in the case of $\mathit{(a)}$, otherwise, we are in the case of $\mathit{(b)}$.
	\end{proof}
\end{prop}

\begin{rem}\label{type-ambig-rem}
	As a result of our definition for Type, we see that the classification is not unique. In particular, there is ambiguity when both entries on the bottom row are invertible. We make this particular choice as it behaves better with respect to irreducibility, see Proposition \ref{red-fil-phi-mod-prop}.
\end{rem}


\subsection{Models for Crystalline Representations}\label{cris-iso-class-sec}

Recall by Theorem \ref{para-equiv-class-thm}, we have a bijection
\[\Theta:\GL_2(\oh_F)^f/\sim_{\kvec}\longrightarrow\SD_{\oh_F\otimes_{\Zp}\oh_K}^{\varphi,\kvec}\]
given by taking an $f$-tuple $(A)=(A^{(i)})$ to a strongly divisible lattice $\Theta(A)=L(A)$ with a basis adapted to its induced filtration so that the $i$-th partial Frobenius action takes the form $\varphi^{(i)}=A^{(i)}\Delta_{\kvec^{(i-1)}}$ with $\Delta_{\kvec^{(i)}}=\Diag(p^{k_i},1)$ by Remark \ref{weight-twist-rem} and with the additional assumption that $k_i>0$ for all $i\in\Z/f\Z$ due to Lemma \ref{para-equiv-killing-lem}. In Section \ref{para-equiv-class-sec}, we observed that we may view each $A^{(i)}$ as being a matrix of one of two Types so that we have effectively chosen simple representatives of the isomorphism classes of strongly divisible lattices.

Recall that we have defined
\[D(A)\coloneqq L(A)\otimes_{\oh_F\otimes_{\Zp}\oh_K}(F\otimes_{\Q_p}K)\]
to be the weakly admissible filtered $\varphi$-module which contains the strongly divisible lattice $L(A)$. All weakly admissible filtered $\varphi$-module necessarily arises out of this construction by Proposition \ref{weak-admis-equiv-SD-prop}. Hence, up to isomorphism, $D(A)=\prod_i D(A)^{(i)}$ admits a basis adapted to the filtration $\{\eta_1^{(i)},\eta_2^{(i)}\}$ such that
\begin{align*}
	[\varphi^{(i)}]_\eta &=\begin{dcases}
		\begin{pmatrix}
			0 & a_1^{(i)} \\ p^{k_{i-1}} & a_2^{(i)}
		\end{pmatrix} & \text{\normalfont Type $\I$} \\
		\begin{pmatrix}
			a_1^{(i)}p^{k_{i-1}} & 0 \\ a_2^{(i)}p^{k_{i-1}} & \alpha^{(i)}
		\end{pmatrix} & \text{\normalfont Type $\II_{\alpha^{(i)}}$}
	\end{dcases} & \Fil^j D(A)^{(i)}&=\begin{dcases}
		D(A)^{(i)} & j\le 0 \\
		F(\eta_1^{(i)}) &  0<j\le k_i \\
		0 & k_i<j.
	\end{dcases}
\end{align*}

\noindent Since we are only concerned with irreducible representations, then we may impose more structure on our partial Frobenius matrices to rid ourselves of some reducible cases and make this description a little bit cleaner.

\begin{prop}\label{red-fil-phi-mod-prop}
	Keeping with notation as above, the weakly admissible filtered $\varphi$-module $D(A)$ is reducible in $\MF_{F\otimes_{\Qp}K}^{\varphi,w.a.}$ in each of the following cases:
	\begin{enumerate}
		\item The $i$-th partial Frobenius $\varphi^{(i)}$ is Type $\II_{\alpha^{(i)}}$ for all $i\in\Z/f\Z$;
		\item The product $\prod_{i\in\Ss} a_2^{(i)}\in p^{w}\oh_F^*$ where $\Ss=\{i\in\Z/f\Z: \varphi^{(i)} \text{ is Type }\I\}$ and $w=\sum_{i\in J}k_i$ for some subset $J\subseteq \Ss$ (if $J=\emptyset$ then $w=0$).
	\end{enumerate}
\end{prop}

\noindent This proposition relies on the following lemma regarding reducibility of strongly divisible lattices in the rank two case.
\begin{lem}\label{eigan-red-lem}
	Continuing to use notations as in Proposition \ref{red-fil-phi-mod-prop}, let $\varphi^f=\prod_{i=0}^{f-1}A^{(i)}\Delta_{\kvec^{(i-1)}}$ be the (linear) Frobenius on $L(A)^{(0)}$. If $\varphi^f$ has an eigenvalue $\lambda$ with $\lambda\in p^w\oh_F^*$ where $w=\sum_{i\in J}k_i$ for some subset $J\subseteq \Z/f\Z$ (if $J=\emptyset$ then $w=0$), then $L(A)$ is reducible in $\SD_{\oh_F\otimes_{\Z_p}\oh_K}^{\varphi}$.
	\begin{proof}
				By assumption there exists simultaneously nonzero $x^{(0)},y^{(0)}\in\oh_F$ such that
		\[\prod_{i=0}^{f-1}A^{(i)}\Delta_{\kvec^{(i-1)}}\begin{pmatrix}
			x^{(0)} \\ y^{(0)}
		\end{pmatrix}=\varphi^f\begin{pmatrix}
			x^{(0)} \\ y^{(0)}
		\end{pmatrix}=\lambda\begin{pmatrix}
			x^{(0)} \\ y^{(0)}
		\end{pmatrix}=\prod_{i=0}^{f-1}\lambda^{(i-1)}\begin{pmatrix}
			x^{(0)} \\ y^{(0)}
		\end{pmatrix}\]
		where each $\lambda^{(i)}\in \oh_F^*$ if $i\notin J$ and $\lambda^{(i)}\in p^{k_i}\oh_F^*$ otherwise. Strong divisibility in the sense of Lemma \ref{strong-div-idem-decomp-lem} then provides the existence of simultaneously non-zero $x^{(i)}, y^{(i)}\in\oh_F$ such that 
		\[A^{(i)}\Delta_{\kvec^{(i-1)}}\begin{pmatrix}
			x^{(i-1)} \\ y^{(i-1)}
		\end{pmatrix}=\lambda^{(i-1)}\begin{pmatrix}
			x^{(i)} \\ y^{(i)}
		\end{pmatrix}\]
		for all $i\in\Z/f\Z$. Hence, each $L(A)^{(i)}$ admits a $\varphi$-stable submodule $L'^{(i)}=x^{(i)}\eta_1^{(i)}+y^{(i)}\eta_2^{(i)}$ so we need only show that $L'=\prod_i L'^{(i)}\subset L(A)$ is strongly divisible using Lemma \ref{strong-div-idem-decomp-lem}.
		
		For any $i\in\Z/f\Z$, observe that since $L'^{(i)}$ is of rank 1, then $\Fil^j L'^{(i)}$ has only one filtration break from $L'^{(i)}$ to $0$. In particular, there is an integer $r \ge0$ so that 
		\[\Fil^j L'^{(i)}=\begin{dcases}
			L'^{(i)} & j\le r \\ 0 & j>r.
		\end{dcases}\]
		Hence, by the arguments above we have
		\[\varphi^{(i)}(\Fil^jL'^{(i-1)})=A^{(i)}\Delta_{\kvec^{(i-1)}}(\Fil^jL'^{(i-1)})=\begin{dcases}
			\lambda^{(i-1)}L'^{(i)} & j\le r \\ 0 & j>r.
		\end{dcases}\]
		However, since $L'^{(i)}=x^{(i)}\eta_1^{(i)}+y^{(i)}\eta_2^{(i)}\subset L(A)^{(i)}$ then strong divisibility of $L(A)^{(i)}$ forces $r=0$ whenever $i-1\notin J$ and $r=k_{i-1}$ when $i-1\in J$. We may then apply Lemma \ref{strong-div-idem-decomp-lem} to deduce strong divisibility of $L'$.
	\end{proof}
\end{lem}

\begin{proof}[Proof of Proposition \ref{red-fil-phi-mod-prop}]
	Let $L(A)=\prod_i L(A)^{(i)}$ be the strongly divisible lattice inside $D(A)=\prod_i D(A)^{(i)}$ with a basis adapted to the filtration $\{\eta_1^{(i)},\eta_2^{(i)}\}$ so that the $i$-th partial Frobenius is given by $\varphi^{(i)}=A^{(i)}\Delta_{\kvec^{(i-1)}}$. Since reducibility is preserved by scalar extension, then we need only display that we can apply Lemma \ref{eigan-red-lem} to each case.
	
	Consider the product $\varphi^f=\prod_{i=0}^{f-1}A^{(i)}\Delta_{\kvec^{(i-1)}}$ and observe that since the left column of each $A^{(i)}\Delta_{\kvec^{(i-1)}}$ is divisible by $p^{k_{i-1}}$ then the lower right entry of $\varphi^f$ must contain the product $\prod_{i,j}a_2^{(i)}\alpha^{(j)}$ for $i\in\Ss$ and $j\notin\Ss$ and any other term on the diagonal is necessarily divisible by $p^{w'}$ where $w'=\sum_{i\in J'} k_i$ for some subset $J'\subseteq \Z/f\Z$. Since the trace of $\varphi^f$ give the sum of the eigenvalues, then one of the eigenvalues takes the form $v=\prod_{i,j}a_2^{(i)}\alpha^{(j)}+p^{w'}\oh_F$. It then becomes clear that we can apply Lemma \ref{eigan-red-lem} to both cases since $\alpha^{(i)}\in \oh_F^*$.
\end{proof}

\noindent We are now ready to give the first of our main results. Recall that $V(A)\coloneqq V^*_\cris(D(A))\in\Rep_{\cris/F}^{\kvec}(G_K)$ is the image of $D(A)$ under the quasi-inverse functor $V^*_\cris$.

\begin{thm}\label{model-irred-cris-reps-thm}
	Let $V\in\Rep_{\cris/F}^{\kvec}(G_K)$ be an irreducible, two dimensional crystalline representation of $G_K$ with labeled Hodge-Tate weights $\{k_i,0\}_{i\in\Z/f\Z}$ where $k_i>0$. Then $V\cong V(A)$ where $D(A)\coloneqq\D^*_\cris(V(A))$ with each $A^{(i)}$ being Type $\I$ or $\II_1$. In particular, $D(A)=\prod_iD(A)^{(i)}$ has a basis $\{\eta_1^{(i)},\eta_2^{(i)}\}$ such that
	\begin{align*}
		[\varphi^{(i)}]_\eta &=\begin{dcases}
			\begin{pmatrix}
				0 & a_1^{(i)} \\ p^{k_{i-1}} & a_2^{(i)}
			\end{pmatrix} & \text{\normalfont Type $\I$} \\
			\begin{pmatrix}
				a_1^{(i)}p^{k_{i-1}} & 0 \\ a_2^{(i)}p^{k_{i-1}} & 1
			\end{pmatrix} & \text{\normalfont Type $\II$}
		\end{dcases} & \Fil^j D(A)^{(i)}&=\begin{dcases}
			D(A)^{(i)} & j\le 0 \\
			F(\eta_1^{(i)}) &  0<j\le k_i \\
			0 & k_i<j.
		\end{dcases}
	\end{align*}
	\begin{proof}
		Let $D=\D_\cris^*(V)$ and observe that Proposition \ref{weak-admis-equiv-SD-prop} and Theorem \ref{para-equiv-class-thm} come together to imply the existence of a $f$-tuple $(B)=(B^{(i)})\in\GL_2(\oh_F)^f$ such that $L(B)$ is a strongly divisible lattice inside $D$ so that $D=D(B)$. By Proposition \ref{class-para-equiv-classes-prop} there is a another $f$-tuple $(A)=(A^{(i)})\in\GL_2(\oh_F)^f$, with each $A^{(i)}$ being Type $\I$ or $\II_{\alpha^{(i)}}$, such that $(B)\sim_{\kvec}(A)$ so that $L(A)\cong L(B)$ by Theorem \ref{para-equiv-class-thm} once more. Hence $D(B)\cong D(A)$ so that $V\cong V(A)$. Since $V$ is irreducible then $D(A)$ is irreducible so that Lemma \ref{red-fil-phi-mod-prop} gives each $A^{(i)}$ the desired Type by Proposition \ref{class-para-equiv-classes-prop} again.
	\end{proof}
\end{thm}

As a result, we will commonly say that the \textit{Type} of $V(A)$ is the Type of the partial Frobenius matrices of $\D^*_\cris(V(A))$. It follows that the $f$-tuple of matrices $(A)$ completely determines the isomorphism class of representations which $V(A)$ is apart of. Hence, to compute reductions, we need only consider those representations $V(A)$ of one of these $2f-1$ Type combinations. We will now write $\II=\II_1$ due to the exclusion of $\II_\alpha$ in the irreducible case for $\alpha\neq 1$.

\begin{rem}\label{recover-f=1-rem}
	It is not difficult to see that our description may be used to recover the $f=1$ classification of Breuil, see \ref{Qp-class-equ}, as we may use irreducibility to exclude the single Type $\II$ case and twist by a suitable unramified character. 
\end{rem}

\begin{rem}\label{type-intersect-rem}
	It should be pointed out that our notion of Type is only a property of strongly divisible lattices. In particular, a weakly admissible filtered $\varphi$-module may admit distinct strongly divisible lattices of different Types. This most obvious reason for this is due to the fact that parabolic equivalence over $F$ need not preserve invertibility. Indeed, Lemma \ref{para-equiv-inv-lem} does not hold over $F$.
\end{rem}

\section{Determining Explicit Kisin Modules}\label{Explicit-Kisin-mod-chap}

\subsection{Initial Frobenius Matrices}\label{inital-frob-sec}

The results of Subsection \ref{BK-chap} gives us a way to explicitly generate a Kisin module over $S_F$ from the data of a weakly admissible filtered $\varphi$-module. In Section \ref{models-irred-reps-chap}, we provided models for irreducible, two-dimensional crystalline representations of $G_K$ in terms of the Type of Frobenius matrix on strongly divisible lattices inside the associated rank-two weakly admissible filtered $\varphi$-modules. We may now use the algorithm described in Remark \ref{kis-alg-rem} to explicitly describe the Frobenius matrix on a Kisin module $\Ms$ over $S_F$ from the strongly divisible lattices described in Section \ref{models-irred-reps-chap}. We briefly recall the algorithm adjusted to our situation below.

Let $(A)=(A^{(i)})\in\GL_2(\oh_F)^f$. For a rank two, weakly admissible filtered $\varphi$-module $D(A)\in\MF_{F\otimes_{\Qp}K}^{\varphi,w.a.,\kvec}$ with each $\kvec^{(i)}=(k_i,0)$ by Remark \ref{weight-twist-rem}, let $\Ds(A)\coloneqq\D_\Bre(D(A))$ denote the corresponding Breuil module over $S_F$. Also let $\Ms(A)\coloneqq\M_\Kis(D(A))\otimes_{\oh_F}S_F$ denote the corresponding finite height Kisin module over $S_F$. By the results of Proposition \ref{BK-iso-prop} and the compatibility of the decomposition $D(A)=\prod_iD(A)^{(i)}$ with respect to the Breuil and Kisin equivalence functors, we have that $\D_\BK(\Ms(A))\cong\Ds(A)$ and we may describe the Frobenius matrix on $\Ms(A)$ in terms of the idempotent decomposition by the following steps:
\begin{enumerate}[label=\arabic*)]
	\item Choose $S_F$-bases $\{e_1^{(i)}, e_2^{(i)}\}$ of $\Ds(A)^{(i)}$ and $\{\alpha_1^{(i)},\alpha_2^{(i)}\}$ of $\Fil^{k_i}\Ds(A)^{(i)}$;
	\item Write $\varphi_{\Ds(A)}^{(i)}(e_1^{(i-1)},e_2^{(i-1)})=(e_1^{(i)},e_2^{(i)})X^{(i)}$ and $(\alpha_1^{(i)},\alpha_2^{(i)})=(e_1^{(i)},e_2^{(i)})B^{(i)}$ for matrices $X^{(i)},B^{(i)}\in\Mat_2(S_F)$;
	\item Then the Kisin module $\Ms(A)$ admits an $S_F$-basis $\{\eta_1^{(i)},\eta_2^{(i)}\}$ with $i$-th partial Frobenius given by $\varphi_{\Ms(A)}^{(i)}(\eta_1^{(i-1)},\eta_2^{(i-1)})=(\eta_1^{(i)},\eta_2^{(i)})\As^{(i)}$ where
	\[\As^{(i)}=E^{k_i}(B^{(i)})^{-1}X^{(i)}\varphi(B^{(i-1)})\varphi(E)^{-k_{i-1}}.\]
\end{enumerate}

\begin{rem}
	To be completely accurate, recall that our Kisin modules are defined over rings previously defined in Section \ref{notation-sec} by  $\Sig_F=(\oh_F\otimes_{\Zp}\oh_K)\pu$, $\Lambda_F\subset(F\otimes_{\Qp}K)\pu$ and $S_F=\Sig[\![\frac{E^p}{p}]\!]\otimes_{\Z_p} F$. When passing to the embedding decomposition, we may use the ring isomorphisms of Lemmas \ref{tensor-decomp-lem} and \ref{int-tensor-decomp-lem} to get Kisin modules which decompose into pieces with coefficients in $\oh_F\pu$ and inside of $F\pu$. We will abuse notation and continue to write $\Sig_F$ and $\Lambda_F$ with the understanding that we are no longer defined over the aforementioned tensor when working with decomposed pieces. We also make the identification $S_F=\Sig_F[\![\frac{E^p}{p}]\!]$ at this level as well. We hope this causes no confusion.
\end{rem}

\begin{prop}\label{construct-kis-mod-prop}
	The finite height Kisin module $\Ms(A)$ admits an $S_F$-basis $\{\eta_1^{(i)},\eta_2^{(i)}\}$ such that $\varphi_{\Ms(A)}^{(i)}(\eta_1^{(i-1)},\eta_2^{(i-1)})=(\eta_1^{(i)},\eta_2^{(i)})\As^{(i)}$ where
	\[\As^{(i)}=\begin{pmatrix}E^{k_i} & 0 \\ 0 & 1\end{pmatrix}A^{(i)}\Delta_{\kvec^{(i-1)}}\begin{pmatrix}\varphi(E)^{-k_{i-1}} & 0 \\ 0 & 1\end{pmatrix}.\]
	\begin{proof}
		To begin, we recall that by definition, the weakly admissible filtered $\varphi$-module $D(A)$ contains a strongly divisible lattice $L(A)$ with a basis adapted to the induced filtration, see Proposition\ref{weak-admis-equiv-SD-prop}. In particular, $D(A)$ admits an $F$-basis $\{e_1^{(i)},e_2^{(i)}\}$ such that the $i$-th partial Frobenius is $\varphi_{D(A)}^{(i)}(e_1^{(i-1)},e_2^{(i-1)})=(e_1^{(i)},e_2^{(i)})A^{(i)}\Delta_{\kvec^{(i-1)}}$ with $\Delta_{\kvec^{(i)}}=\Diag(p^{k_i},1)$.
		
		For each $e_j^{(i)}\in D(A)$, set $\hat{e}_j^{(i)}\coloneq1_S\otimes e_j^{(i)}\in\Ds(A)$. Then $\Ds(A)$ is a free $S_F$-module with basis $\{\hat{e}_1^{(i)},\hat{e}_2^{(i)}\}$ such that
		\[\varphi_{\Ds(A)}^{(i)}(\hat{e}_1^{(i-1)},\hat{e}_2^{(i-1)})=(\hat{e}_1^{(i)},\hat{e}_2^{(i)})A^{(i)}\Delta_{\kvec^{(i-1)}}.\]
		Since $D(A)$ has trivial monodromy then it is easy to compute $\Fil^{k_i}(\Ds(A)^{(i)})=S_F\hat{e}_1^{(i)}\oplus S_FE^{k_i}\hat{e}_2^{(i)}$. Hence, we may choose the basis $\{\hat{e}_1^{(i)},E^{k_i}\hat{e}_2^{(i)}\}$ of $\Fil^{k_i}(\Ds(A)^{(i)})$ so that the matrix $B$ of step $(2)$ above is given by $(B)=(B^{(i)})$ where $B^{(i)}=\Diag(1,E^{k_i})$. A brief matrix computation then gives us that the matrix $(\As)=(\As^{(i)})$ of step $(3)$ is as described. 
	\end{proof}
\end{prop}

We can now explicitly describe the Frobenius matrices $\As^{(i)}$ on a Kisin module $\Ms(A)=\prod_i \Ms(\As)^{(i)}$ in terms of the matrix Types of $(A^{(i)})$. Recall $\gamma=\varphi(E)/p\in S_F^*$ so that $\varphi(E)=p\gamma$.
\begin{itemize}
	\item Suppose $A^{(i)}$ is of Type $\I$. We may then write $A^{(i)}=\begin{psmallmatrix}0 & a_1^{(i)} \\ 1 & a_2^{(i)}\end{psmallmatrix}$ so that 
	\[\As^{(i)}=\begin{pmatrix}0 & E^{k_i}a_1^{(i)} \\ \gamma^{-k_{i-1}} & a_2^{(i)}\end{pmatrix};\]
	\item Suppose $A^{(i)}$ is of Type $\II$. We may then write $A^{(i)}=\begin{psmallmatrix}a_1^{(i)} & 0 \\ a_2^{(i)} & 1\end{psmallmatrix}$ so that 
	\[\As^{(i)}=\begin{pmatrix}\gamma^{-k_{i-1}}E^{k_i}a_1^{(i)} & 0 \\ \gamma^{-k_{i-1}}a_2^{(i)} & 1\end{pmatrix}.\]
\end{itemize}

\noindent Since the partial Frobenius matrices $(\As^{(i)})$ on $\Ms(A)$ are completely determined by the $f$-tuple $(A^{(i)})$, we will commonly write $\Ms(A)=\Ms(\As)$ with the understanding that $(\As)$ comes from $(A)$ which in turn comes from some strongly divisible lattice $\Theta(A)$. We will call each $k_i$ a \textit{labeled height} corresponding to $\Ms(\As)^{(i)}$.

\subsection{Change of Basis}\label{change-of-basis-sec}
From the previous section, we explicitly defined finite height, rank-two Kisin modules over $S_F$ arising from irreducible, weakly admissible filtered $\varphi$-modules. The purpose of this section is to change base so that the determinants of partial Frobenius are in a particular form. This special basis will be a critical piece of the puzzle when we later compute reductions. 

For a base change $f$-tuple $(X^{(i)})$, we will write $(X^{(i)})*_\varphi(\As^{(i)})=(\Bs^{(i)})$ to denote
\[\Bs^{(i)}=X^{(i)}\cdot\As^{(i)}\cdot\varphi(X^{(i-1)})^{-1}\]
for all $i\in\Z/f\Z$. As we are now dealing with $\varphi$-twisted base changes, we would like to first introduce some simplifying notation. For $f(\varphi)=\sum c_j\varphi^j\in\Z[\varphi]$ we will write $g^{f(\varphi)}=\prod_j \varphi^j(g)^{c_i}$. Recall $\gamma=\varphi(E)/p\in S_F^*$.

\begin{lem}\label{det-base-change-system-lem}
	For $f(\varphi)\in\Z[\varphi]$, consider the equation $x^{\varphi^b-1}=\gamma^{-f(\varphi)}$ over $S_F$ for some  integer $b>0$. Then there exists a solution $x=\lambda_b^{f(\varphi)}$ where
	\[\lambda_b=\prod_{n\ge0}\varphi^{bn}(\gamma)\in S_F^*.\]
	\begin{proof}
		It suffices to show that $\varphi^b(\lambda_b^{f(\varphi)})=\gamma^{-f(\varphi)}\lambda_b^{f(\varphi)}$. The right hand side of the equation gives
		\[\gamma^{-f(\varphi)}\lambda_b^{f(\varphi)}=\Bigg(\frac{p^{f(\varphi)}}{\varphi(E)^{f(\varphi)}}\Bigg)\Bigg(\prod_{n\ge0}\varphi^{bn}\left(\frac{\varphi(E)^{f(\varphi)}}{p^{f(\varphi)}}\right)\Bigg)=\prod_{n\ge1}\varphi^{bn}\left(\frac{\varphi(E)^{f(\varphi)}}{p^{f(\varphi)}}\right)=\varphi^b(\lambda_b^{f(\varphi)})\]
		The fact that $\lambda_b\in S_F^*$ follows form the fact that $\gamma\in S_F^*$.
	\end{proof}
\end{lem}

\noindent Recall that we have decided to write $(xy)^{(i)}=x^{(i)}y^{(i)}$.

\begin{prop}\label{det-base-change-existance-prop}
	There exists an $f$-tuple $(X^{(i)})\in\GL_2(S_F)^f$ such that $(X^{(i)})*_\varphi(\As^{(i)})=(\Bs^{(i)})$ with $\det(\Bs^{(i)})\in E^{k_i}\oh_F^*$ for all $i\in\Z/f\Z$. 
	\begin{proof}
		Consider the matrices $X^{(i)}=\Diag(x_1^{(i)},x_2^{(i)})$ for all $i\in\Z/f\Z$. Computing $(X^{(i)})*_\varphi(\As^{(i)})$ in each embedding gives two cases depending on the Type of $\As^{(i)}$:
		\begin{center}
			Type $\I$ : $\begin{pmatrix}
				0 & \frac{E^{k_i}(a_1x_1)^{(i)}}{\varphi(x_2^{(i-1)})} \\ \frac{x_2^{(i)}}{\gamma^{k_{i-1}}\varphi(x_1^{(i-1)})} & \frac{(a_2x_2)^{(i)}}{\varphi(x_2^{(i-1)})}
			\end{pmatrix}$ \hspace{1cm} Type $\II$ : $\begin{pmatrix}
				\frac{E^{k_i}(a_1x_1)^{(i)}}{\gamma^{k_{i-1}}\varphi(x_1^{(i-1)})} & 0 \\ \frac{(a_2x_2)^{(i)}}{\gamma^{k_{i-1}}\varphi(x_1^{(i-1)})} & \frac{x_2^{(i)}}{\varphi(x_2^{(i-1)})}
			\end{pmatrix}$.
		\end{center}
		In the case of Type $\I$, we set
		\begin{align*}
			x_1^{(i)}&=\varphi(x_2^{(i-1)})&x_2^{(i)}&=\gamma^{k_{i-1}}\varphi(x_1^{(i-1)}),
		\end{align*}
		and in the case of Type $\II$, we set
		\begin{align*}
			x_1^{(i)}&=\gamma^{k_{i-1}}\varphi(x_1^{(i-1)})&x_2^{(i)}&=\varphi(x_2^{(i-1)}).
		\end{align*}
		Hence, depending on the Type of each $\As^{(i)}$, we get a recursive system of $2f$ equations. When there is an even number of Type $\I$ matrices, the system admits two disjoint $f$-cycles of which $x_1^{(i)}$ is apart of one and $x_2^{(i)}$ is apart of the other. On the other hand, when the number of Type $\I$ matrices is odd, then the system consists of a single $2f$-cycle.\footnote{This is the first sign that there is a dichotomy of behavior depending on the even or oddness of the number of Type $\I$ matrices in the Type combination. Indeed, we will see that the shape of the reductions depend on this as well.}
		
		Solving for $x_1^{(i)}$ and $x_2^{(i)}$ will give 
		\begin{align*}
			(x_1^{(i)})^{\varphi^b-1}&=\gamma^{-g^{(i)}(\varphi)}&(x_2^{(i)})^{\varphi^b-1}&=\gamma^{-h^{(i)}(\varphi)}
		\end{align*}
		for some $g^{(i)}(\varphi),h^{(i)}(\varphi)\in\Z[\varphi]$. When there is an even number of Type $\I$ matrices, the two disjoint $f$-cycles forces $b=f$  and when there is an odd number of Type $\I$, the single $2f$-cycle means $b=2f$. 
		
		By Lemma \ref{det-base-change-system-lem}, there exists solutions $x_1^{(i)}=\lambda_b^{g^{(i)}(\varphi)}$ and $x_2^{(i)}=\lambda_b^{h^{(i)}(\varphi)}$ in $S_F^*$. The result is a matrix $\Bs^{(i)}$ given by
		\begin{itemize}
			\item Type $\I$: \[\Bs^{(i)}=\begin{pmatrix}
				0 & E^{k_i}a_1^{(i)} \\ 1 & a_2^{(i)}\lambda_b^{h^{(i)}(\varphi)-g^{(i)}(\varphi)}
			\end{pmatrix};\]
			\item Type $\II$: \[\Bs^{(i)}=\begin{pmatrix}
				E^{k_i}a_1^{(i)} & 0 \\ a_2^{(i)}\lambda_b^{h^{(i)}(\varphi)-g^{(i)}(\varphi)} & 1
			\end{pmatrix}.\]
		\end{itemize}
		Hence, we see $\det(\Bs^{(i)})=E^{k_i}a_1^{(i)}\in E^{k_i}\oh_F^*$ for each $i\in\Z/f\Z$ as desired.
	\end{proof}
\end{prop}

\section{A Descent Algorithm for Kisin Modules}\label{red-alg-chap}

\subsection{Overview of our Methods}\label{methods-sec}
Section \ref{Explicit-Kisin-mod-chap} has given us an explicit description of finite height, rank-two Kisin modules over $S_F$ arising from irreducible, weakly admissible filtered $\varphi$-modules. In order to  use these Kisin modules to compute reductions of the associated crystalline representations, we will need to \textit{descend}, in the sense of Definition \ref{descent-def}, the coefficients of our Kisin modules from $S_F$ to $\Sig_F$ where reduction modulo $\varpi$ is possible as detailed in Subsection \ref{BK-chap}. The methods displayed in this article to do so are heavily inspired by Liu's reduction algorithm in \cite{Liu21} which are in turn, inspired by the algorithms of Bergdall and Levin in \cite{BL20} and Bergdall, Levin, and Liu in \cite{BLL22}. The common trope between all of these is a two step process that begins with `preparation' computations followed by a theoretic `descent' stage. In simple terms, one is typically forced to begin the process with a Kisin module over enlarged coefficients that are not able to be directly reduced, such as $S_F=\Sig_F[\![\frac{E^p}{p}]\!]$. A preparation stage is used to find a basis where we can split our matrix into an integral $\Sig_F$ piece and a more complicated piece that is sufficiently small in some adic sense inside $S_F$. The descent stage then justifies why we can descend our coefficients so that the more complicated piece is disregarded and we are left only with the integral piece. Our methods follow this recipe closely.

Recall that by Section \ref{change-of-basis-sec}, we have achieved an explicit description of Kisin modules $\Ms(\As)$ of finite height over $S_F$ by describing their $i$'th partial Frobenius actions by the matrix $\As^{(i)}$ in terms of Types given by:
\begin{itemize}
	\item Type $\I$: \[\As^{(i)}=\begin{pmatrix}
		0 & E^{k_i}a_1^{(i)} \\ 1 & a_2^{(i)}\lambda_b^{h^{(i)}(\varphi)-g^{(i)}(\varphi)}
	\end{pmatrix};\]
	\item Type $\II$: \[\As^{(i)}=\begin{pmatrix}
		E^{k_i}a_1^{(i)} & 0 \\ a_2^{(i)}\lambda_b^{h^{(i)}(\varphi)-g^{(i)}(\varphi)} & 1
	\end{pmatrix}.\]
\end{itemize}

A quick glance at our coefficients make it obvious that reduction modulo $\varpi$ is not possible for such Kisin modules. In Section \ref{prep-comp-chap}, we will (under some restrictions) explicitly construct a change of basis $X$ over $\GL_2(S_F[1/p])$ so that $X*_\varphi\As=\Af_0+C$ with $\Af_0\in\GL_2(\Sig_F)$ and $C\in\Mat_2(I)$ where $I\subset S_F$ is a particularly chosen ideal which is sufficiently large in both a $\varpi$ and $E$-adic sense. This  will constitute our `preparation' stage and involves fairly delicate computations to ensure that elements are either integral or are adically small enough to lie in $I$.

In the following section, we will prove some results making up the `descent' stage of our algorithm. The idea will be that if we assume that the basis achieved from the preparation stage above exists, then we can prove the existence of another change of basis $Y$ over $R[1/p]=\Sig_F[\![\varpi E/p]\!][1/p]$ so that $Y*_\varphi\As=\Af$ where $\Af\in\GL_2(\Sig_F)$ and moreover, we will have $\Af\equiv \Af_0\pmod\varpi$. Hence, we will have descended the coefficients of $\As$ to $\Sig_F$ and have a complete understanding of the reduction modulo $\varpi$. We note that this stage of the algorithm is not explicit and indeed, we have little knowledge of what $\Af$ actually looks like; however, we know exactly what its reduction of Frobenius is and this is enough for our purposes.

\subsection{The Descent Algorithm}\label{descent-alg-sec}
As told in Section \ref{methods-sec}, the descent stage of our reduction algorithm will depend on the existence of a certain base change which allows us to split each partial Frobenius $\As^{(i)}$ into a reducible part and a sufficiently adically small part. Hence, before stating the descent algorithm, we will need to establish a set of assumptions we will call the \textit{descent assumptions} under which descent may be carried out.

Let us begin with two very important sets that will give us our notion of elements that are `sufficiently adically small'. Recall that $\Fil^j S_F=E^j S_F$ and let $\ident_d$ denote the usual $d\times d$ identity matrix.

\begin{defin}\label{I-J-def}
	For an integer $c>1$, define the sets $I_c\coloneqq \varpi p^{c-1} S_F+\varpi\Fil^{cp}S_F+E\Fil^{cp}S_F\subset S_F$ and $J_c=\varpi\Fil^{cp}S_F+E\Fil^{cp}S_F\subset S_F$.
\end{defin}

\begin{rem}\label{EFil-rem}
	The presence of $E\Fil^{cp}S_F$ in both $I_c$ and $J_c$ is not required for the purposes of this article. However, they are critical to \cite{Guz2} which utilizes this construction so we include it here for completeness.
\end{rem}

\noindent  Hence, $I_c=\varpi p^{c-1}S_F+J_c$ and it is not difficult to see that both $I_c,J_c\subset S_F$ are ideals. We are now ready to state the descent assumptions that will be verified to hold under certain restrictions in Section \ref{prep-comp-chap}.

\begin{assump}[Descent Assumptions]\label{descent-assump}
	Let $\Ms(\As)=\prod_i\Ms(\As)^{(i)}$ be a rank-two Kisin module over $S_F$ with labeled heights $k_i$. Suppose there exists an integer $c>1$ and a sequence of base changes $(X_n)$ so that by setting $(X_n^{(i)})*_\varphi(\As^{(i)})=(\As_n^{(i)})$, the following holds:
	\begin{enumerate}
		\item We have $c(p-2)\ge k_i$ for all $i\in\Z/f\Z$;
		\item Either $X_n^{(i)}\in\GL_2(S_F[1/p])$ with $\det(X_n^{(i)})=1$ or $X_n^{(i)}=\Diag(x,y)\in\GL_2(F)$;
		\item There exists a finite $m>0$ such that 
		\[\As_m^{(i)}=\Af_0^{(i)}+C^{(i)}\]
		where $\Af_0^{(i)}\in\Mat_2(\Sig_F)$ and $C^{(i)}\in\Mat_2(I_c)$.
	\end{enumerate}
\end{assump}

Let us state the descent theorem and temporarily forgo the proof until we have built up the needed results. Recall that $S_F=\Sig_F[\![E^p/p]\!]$ and define an auxiliary ring $R=\Sig_F[\![\varpi E/p]\!]$. For any ring $\Ss$ such that $\Sig_F\subset \Ss\subset F[\![u]\!]$ and $h\ge0$, define the subset of \textit{height-$h$} matrices:
\[\Mat_d^{h}(\Ss)\coloneqq\big\{A\in\Mat_d(\Ss):\text{\normalfont there exists}\hspace{.2cm} B\in\Mat_d(\Ss) \hspace{.2cm}\text{\normalfont with}\hspace{.2cm} AB=BA=E^h\ident_d\big\}.\]

We are now able state our main result of this section, the theoretical foundation which permits descent. The following theorem may be viewed as a version of \cite[Prop. 3.3]{Liu21} in the dimension two and $f>1$ case.

\begin{thm}\label{descent-thm}
	Suppose the Descent Assumptions \ref{descent-assump}. Then there exists a base change $(Y^{(i)})\in\GL_2(R[1/p])^f$ such that $Y^{(i)}*_\varphi\As^{(i)}=\Af^{(i)}\in\Mat_2^{k_i}(\Sig_F)$ with $\Af^{(i)}\equiv\Af_0^{(i)}\pmod\varpi$.
\end{thm}

Of critical importance to our reduction process will be the following algorithm. We note that it is stated in far more generality then what is permitted by the Descent Assumptions \ref{descent-assump}. Indeed, we make no restriction on dimension or height. This may be viewed as a generalization of \cite[Prop 3.2]{Liu21} to the $f>1$ case. Note that it is at this point where we begin to exclude $p=2$ from the proceedings.

\begin{prop}\label{red-alg-prop}
	Let $p>2$ and consider an $f$-tuple of matrices $(A^{(i)})$ with each $A^{(i)}\in\Mat_d^{h^{(i)}}(S_F)$. Suppose we may write $A^{(i)}=A_1^{(i)}+C_1^{(i)}$ with $A_1^{(i)}\in\Mat_d^{h^{(i)}}(\Sig_F)$ and $C_1^{(i)}\in\Mat_d(\varpi\Fil^{h_1^{(i)}}S_F)$ where $h_1^{(i)}\ge \frac{\max_i(h^{(i)})}{1-2p^{-1}}$ for all $i\in\Z/f\Z$. Then there exists a sequence of base change matrices $(Y_n^{(i)})\in\Mat_d(R)$ such that:
	\begin{itemize}
		\item As $n\rightarrow\infty$, $Y_n^{(i)}$ converges to $Y^{(i)}$ in $\GL_d(R)$;
		\item $Y^{(i)}*_\varphi A^{(i)}=\Af^{(i)}\in \Mat_d^{h^{(i)}}(\Sig_F)$;
		\item $\Af^{(i)}\equiv A_1^{(i)}\pmod\varpi$.
	\end{itemize}
	\begin{proof}
		Without loss of generality, let us fix an $i\in\Z/f\Z$ and construct the sequence of matrices $Y_n^{(i)}$ by successive approximation. By assumption we know that $C_1^{(i)}\in\Mat_d(\varpi\Fil^{h_1^{(i)}}S_F)$ so we may write $C_1^{(i)}=C_1'^{(i)}E^{h^{(i)}}$ since $h_1^{(i)}>h^{(i)}$. As $A_1^{(i)}\in\Mat_d^{h^{(i)}}(\Sig_F)$ then there exists $B_1^{(i)}\in\Mat_d^{h^{(i)}}(\Sig_F)$ such that $C_1^{(i)}=C_1'^{(i)}E^{h^{(i)}}=C_1'^{(i)}B_1^{(i)}A_1^{(i)}$. Setting $Y_0^{(i)}=\ident_d$, we may write
		\[(Y_{0}^{(i)})*_\varphi(A^{(i)})=A_1^{(i)}+C_1^{(i)}=A_1^{(i)}+C_1'^{(i)}B_1^{(i)}A_1^{(i)}=\left(\ident_d+C_1'^{(i)}B_1^{(i)}\right)A_1^{(i)}\]
		We will set $Y_1^{(i)}\coloneqq (\ident_d+C_1'^{(i)}B_1^{(i)})^{-1}Y_{0}^{(i)}$ and observe that
		\begin{align*}
			(Y_1^{(i)})*_\varphi(A^{(i)}) &= (\ident_d+C_1'^{(i)}B_1^{(i)})^{-1}Y_{0}^{(i)}\cdot A^{(i)}\cdot\varphi(Y_{0}^{(i-1)})^{-1}(\ident_d+\varphi(C_1'^{(i-1)}B_1^{(i-1)})) \\
			&= (\ident_d+C_1'^{(i)}B_1^{(i)})^{-1}\left(\ident_d+C_1'^{(i)}B_1^{(i)}\right)A_1^{(i)}(\ident_d+\varphi(C_1'^{(i-1)}B_1^{(i-1)})) \\
			&=A_1^{(i)}(I+\varphi(C_1'^{(i-1)}B_1^{(i-1)})).
		\end{align*}
		
		\noindent We now embark on the task of writing $A_1^{(i)}(I+\varphi(C_1'^{(i-1)}B_1^{(i-1)}))=A_2^{(i)}+C_2^{(i)}$ so we can repeat the above process.
		
		Observe that any element of $S_F$ may be written as $\sum_{j\ge0}a_j\frac{E^j}{p^{\lfloor j/p\rfloor}}$ where $a_j\in\oh_F[u]$. Since $C_1^{(i)}=C_1'^{(i)}E^{h^{(i)}}$ then any element of $C_1'^{(i)}$ will take the form $x=\sum_{j\ge h_1^{(i)}}a_j\frac{E^{j-h^{(i)}}}{p^{\lfloor j/p \rfloor}}$ where each $a_j\in \varpi\Sig_F$. We have $\varphi(E)=p\gamma$ with $\gamma\in S_F^*$ so that $\varphi(x)\in\varpi p^{\ell_1^{(i)}}S_F$ where $\ell_1^{(i)}=h_1^{(i)}-h^{(i)}-\lfloor\frac{h_1^{(i)}}{p}\rfloor$. Hence, we can write $\varphi(x)=\sum_{j\ge0} b_j p^{\ell_1^{(i)}}\frac{E^j}{p^{\lfloor j/p\rfloor}}=y+z$ where $y\in\varpi\Sig_F\cap p^{\ell_1^{(i)}}S_F$ and $z\in\varpi\Fil^{p(\ell_1^{(i)}+1)}S_F$. We set $h_{2}^{(i)}=p(\ell_1^{(i)}+1)$. In fact, we have $h_{2}^{(i)}>h_1^{(i)}$ which can be seen by noticing that proving the previous inequality is equivalent to proving the following inequality after expanding $\ell_1^{(i)}$,
		\[h_1^{(i)}(1-p^{-1})-\Big\lfloor\frac{h_1^{(i)}}{p}\Big\rfloor+1>h^{(i)}.\]
		Hence, we may write $h_1^{(i)}(1-p^{-1})-\lfloor\frac{h_1^{(i)}}{p}\rfloor+1>h_1^{(i)}(1-2p^{-1})\ge\max_i(h^{(i)})\ge h^{(i)}$ which gives the desired conclusion since $p>2$.
		
		Since $B_1^{(i)}\in\Mat_d^{h^{(i)}}(\Sig_F)$ then the above discussion implies that $\varphi(C_1'^{(i)}B_1^{(i)})=D_{2}^{(i)}+D_{2}'^{(i)}$ with $D_{2}^{(i)}\in\Mat_d(\varpi\Sig_F\cap p^{\ell_1^{(i)}}S_F)$ and $D_{2}'^{(i)}\in\Mat_d(\varpi\Fil^{h_{2}^{(i)}}S_F)$. Let us set
		\begin{align*}
			A_{2}^{(i)}&=A_1^{(i)}(\ident_d+D_{2}^{(i-1)}) & C_{2}^{(i)}=A_1^{(i)}D_{2}'^{(i-1)}.
		\end{align*}
		
		\noindent Since $D_{2}^{(i)}\in\Mat_d(\varpi\Sig_F\cap p^{\ell_1^{(i)}}S_F)$ the geometric series $\sum_{j\ge0}(-D_{2}^{(i)})^j=(\ident_d+D_{2}^{(i)})^{-1}$ converges in $\Mat_d(\Sig_F)$. Hence, if $B_1^{(i)}\in\Mat_d^{h^{(i)}}(\Sig_F)$ is the matrix such that $A_1^{(i)}B_1^{(i)}=E^{h^{(i)}}\ident_d$ then by setting $B_{2}^{(i)}=(I+D_{2}^{(i-1)})^{-1}B_1^{(i)}$ we get that $A_{2}^{(i)}\in\Mat_d^{h^{(i)}}(\Sig_F)$. We also notice that by construction it is clear that $C_{2}^{(i)}\in\Mat_d(\varpi\Fil^{h_{2}^{(i-1)}}S_F)$. Since $h_{2}^{(i-1)}>h_1^{(i-1)}>\max_i(h^{(i)})\ge h^{(i)}$ then we can repeat this process $(f-1)$-times to get a base change matrix $Y_{f}^{(i)}=\prod_{j=0}^{f}(\ident_d+C_j'^{(i)}B_j^{(i)})^{-1}$ so that 
		\[Y_{f}^{(i)}*_\varphi A^{(i)}=A_{f+1}^{(i)}+C_{f+1}^{(i)}\] 
		where $A_{f+1}^{(i)}\in\Mat_d^{h^{(i)}}(\Sig_F)$ and $C_{f+1}^{(i)}\in\Mat_d(\varpi\Fil^{h_{f+1}^{(i)}}S_F)$ where $h_{f+1}^{(i)}>h_{1}^{(i)}$. Hence, we have returned to the original assumptions with the only difference being our $C_{f+1}^{(i)}$ matrix is $E$-adically smaller than $C_1^{(i)}$. 
		
		Let us set $Y^{(i)}=\lim\limits_{n\rightarrow\infty}Y_n^{(i)}=\prod_{j=0}^{\infty}(\ident_d+C_j'^{(i)}B_j^{(i)})^{-1}$. Recall that we have written $Y_n^{(i)}=(\ident_d+C_n'^{(i)}B_n^{(i)})^{-1}Y_{n-1}^{(i)}$ and that we may denote any element of the matrix $C_n'^{(i)}B_n^{(i)}$ by $x=\sum_{j\ge h_n^{(i)}}a_j\frac{E^{j-h^{(i)}}}{p^{\lfloor j/p \rfloor}}$ with $a_j\in\varpi\Sig_F$. Since $h_n^{(i)}>\max{(h^{(i)})}\ge h^{(i)}$ then $j-h^{(i)}\ge\lfloor j/p \rfloor$ for any $j\ge h_n^{(i)}$. Hence $x\in R$ and $(\ident_d+C_n'^{(i)}B_n^{(i)})^{-1}\in \GL_d(R)$ so as $n\rightarrow\infty$, $Y^{(i)}=\prod_{j=0}^{\infty}(\ident_d+C_j'^{(i)}B_j^{(i)})^{-1}\in\GL_d(R)$ converges as desired proving (a) since $C_j'^{(i)}B_j^{(i)}\rightarrow 0$ in the $E$-adically closed ring $R$.
		
		Finally, observe that by construction we have that $A_{n+1}^{(i)}\equiv A_n^{(i)}\pmod\varpi$ and $A_{n+1}^{(i)}-A_n^{(i)}\in \Mat_d(p^{\ell_n^{(i+1)}}S_F)$ so as $n\rightarrow\infty$, we have that $A_n^{(i)}$ converges to some $\mathfrak{A}^{(i)}\in\Mat_d^{h^{(i)}}(\Sig_F)$ with $\mathfrak{A}^{(i)}\equiv A_1^{(i)}\pmod\varpi$. Moreover, as $n\rightarrow\infty$, the above analysis shows that $h_n^{(i)}\rightarrow\infty$ so that $C_n^{(i)}\rightarrow 0$ in the $E$-adically closed ring $S_F$. This proves both $(b)$ and $(c)$.
	\end{proof}
\end{prop}

To prove Theorem \ref{descent-thm}, we will need to justify the use of Proposition \ref{red-alg-prop} on $\As_m^{(i)}=\Af_0^{(i)}+C^{(i)}$. Such a justification requires the following two lemmas.

\begin{lem}\label{I-decomp-lem}
	For any element $x\in I_c$, we may write $x\in\varpi\Sig_F+J_c$. Moreover, by assuming the Descent Assumptions \ref{descent-assump} and decomposing $C^{(i)}\in\Mat_2(I_c)$ in this way, then $\As_m^{(i)}=\Af'^{(i)}+C'^{(i)}$ where $\Af'^{(i)}\in\Mat_2(\Sig_F)$ and $C'^{(i)}\in\Mat_2(J_c)$ with $\Af'^{(i)}\equiv\Af_0\pmod\varpi$ and  $\det(\Af'^{(i)})\in E^{k_i}\Sig_F^*$.
	\begin{proof}
		Since both $\varpi\Fil^{cp}S_F, E\Fil^{cp}S_F\in J$ then we need only show $\varpi p^{c-1} S_F\in \varpi\Sig_F+J_c$. This is easily seen once we notice that any element of $\varpi p^{c-1}S_F$ takes the form
		\[\varpi p^{c-1}\sum_{j=0}^{\infty}\alpha_j\frac{E^{jp}}{p^j}=\varpi \left(\sum_{j=0}^{c-1}p^{c-1-j}\alpha_j E^{jp}\right)+\varpi\Fil^{cp}S_F \in \varpi\Sig_F+\varpi\Fil^{cp}S_F\subset \varpi\Sig_F+J_c.\]
		
		By the first part of this lemma, we need only show that $\det(\Af'^{(i)})\in E^{k_i}\Sig_F^*$. By the definitions of $(X_n^{(i)})$, we know that $\det(\As_m^{(i)})=\det(\As^{(i)})= E^{k_i}a_1^{(i)}$. It is easy to see that $\det(\Af'^{(i)})= E^{k_i}a_1^{(i)}\pmod{\Fil^{cp}S_F}$ and hence, $\det(\Af'^{(i)})-E^{k_i}a_1^{(i)}\in \Fil^{cp}S_F\cap \Sig_F=E^{cp}\Sig_F$. Since each $k_i<cp$ then $\det(\Af'^{(i)})\in E^{k_i}a_1^{(i)}+E^{cp}\Sig_F\in E^{k_i}\Sig_F^*$.
	\end{proof}
\end{lem}

\begin{lem}\label{frob-h-mat-lem}
	We have $\As^{(i)}\in\Mat_2^{k_i}(S_F)$ and $\As_m^{(i)}\in\Mat_2^{k_i}(S_F[1/p])$.
	\begin{proof}
		The fact that $\As^{(i)}\in\Mat_2^{k_i}(S_F)$ is clear. For the second part of the lemma, observe that \[\As_m^{(i)}=X_m^{(i)}X_{m-1}^{(i)}\cdots X_1^{(i)}\As\varphi(X_1^{(i)}\cdots X_{m-1}^{(i)}X_m^{(i)})^{-1}.\]
		Hence, by setting \[B_m^{(i)}=\varphi(X_m^{(i)}\cdots X_1^{(i)})B^{(i)}(X_1^{(i)}\cdots X_{m-1}^{(i)}X_m^{(i)})^{-1}\] then the fact that $\As_m^{(i)}\in\Mat_2^{k_i}(S_F[1/p])$ follows from the first part of the lemma and the fact that $\As^{(i)}B^{(i)}=B^{(i)}\As^{(i)}=E^{k_i}\ident_2$ in $\Mat_2(S_F)$.
	\end{proof}
\end{lem}

\begin{proof}[Proof of Theorem \ref{descent-thm}]
	Our plan is to justify the use of Proposition \ref{red-alg-prop} on $\As_m^{(i)}$ with $h_1^{(i)}=cp$. At which point, the existence of $(Y^{(i)})$ with the desired properties will be implied.
	
	We begin by invoking Lemma \ref{I-decomp-lem} in order to write $\As_m^{(i)}=\Af'^{(i)}+C'^{(i)}$ where $\Af'^{(i)}\in\Mat_2(\Sig_F)$ and $C'^{(i)}\in\Mat_2(J_c)$ with $\Af'^{(i)}\equiv\Af_0\pmod\varpi$ and  $\det(\Af'^{(i)})= E^{k_i}\alpha$ for some $\alpha\in\Sig_F^*$. To see that $\Af'^{(i)}\in\Mat_2^{k_i}(\Sig_F)$, we need only define $B'^{(i)}=\alpha^{-1}\det(\Af'^{(i)})(\Af'^{(i)})^{-1}\in\Mat_2(\Sig_F)$. A brief matrix computation will show that $\Af'^{(i)}B'^{(i)}=B'^{(i)}\Af'^{(i)}=E^{k_i}\ident_2$ so that $\Af'^{(i)}\in\Mat_2^{k_i}(\Sig_F)$.
	
	The only hold up to the application of Proposition \ref{red-alg-prop} lies in the fact that $C'^{(i)}\in\Mat_2(J_c)$ and $J_c$ is not in $\varpi\Fil^{cp}S_F$. To do this, we follow similar ideas to the proof of the aforementioned proposition. We may decompose $C'^{(i)}=\tilde{C}^{(i)}E^{k_i}\ident_2$ as $k_i<cp$. Hence by the above analysis, we may write $C'^{(i)}=\tilde{C}^{(i)}E^{k_i}\ident_2=\tilde{C}^{(i)}B'^{(i)}\Af'^{(i)}$. Let $Z^{(i)}=(\ident_2+\tilde{C}^{(i)}B'^{(i)})^{-1}\in\GL_2(S_F)$ and observe that
	\begin{align*}
		(Z^{(i)})*_\varphi(A_m^{(i)}) &= (\ident_2+\tilde{C}^{(i)}B'^{(i)})^{-1}\cdot A_m^{(i)}\cdot(\ident_2+\varphi(\tilde{C}^{(i-1)}B'^{(i-1)})) \\
		&= (\ident_2+\tilde{C}^{(i)}B'^{(i)})^{-1}(\ident_2+\tilde{C}^{(i)}B'^{(i)})\Af'^{(i)}(\ident_2+\varphi(\tilde{C}^{(i-1)}B'^{(i-1)})) \\
		&=\Af'^{(i)}(\ident_2+\varphi(\tilde{C}^{(i-1)}B'^{(i-1)})).
	\end{align*}
	
	Since $C'^{(i)}\in\Mat_2(J_c)$ then any element of $\tilde{C}^{(i)}$ will take the form $x=\sum_{j\ge cp}a_j\frac{E^{j-k_i}}{p^{\lfloor j/p\rfloor}}$ with $a_j\in (\varpi,E)\Sig_F$. Since $\varphi(E)=p\gamma\in p S_F$ and $k_i\le c(p-2)$ then $j-k_i\ge 2c$ so that $\varphi(E)^{j-k_i}\in p^{2c}S_F$. Hence, $\varphi(\tilde{C}^{(i)})\in\Mat_2(\varpi p^{c-1} S_F)$ so we decompose $\varphi(\tilde{C}^{(i)})=D_1^{(i)}+D_2^{(i)}$ where $D_1^{(i)}\in\Mat_2(\varpi\Sig_F)$ and $D_2^{(i)}\in\Mat_2(\varpi\Fil^{cp}S_F)$. By setting
	\begin{align*}
		A_1^{(i)}&=\Af'^{(i)}(\ident_2+D_1^{(i-1)}\varphi(B'^{(i-1)}))&C_1^{(i)}&=\Af'^{(i)}D_2^{(i-1)}\varphi(B'^{(i-1)}),
	\end{align*}
	
	\noindent it is easy to see that $A_1^{(i)}\in\Mat_2^{k_i}(\Sig_F)$ with $A_1^{(i)}\equiv\Af'^{(i)}\pmod\varpi$ and $C_1^{(i)}\in\Mat_2(\varpi\Fil^{cp}S_F)$. Since $k_i\le c(p-2)$ then setting $h_1^{(i)}=cp$ for all $i\in\Z/f\Z$ allows us to use Proposition \ref{red-alg-prop} on $\As_m^{(i)}=A_1^{(i)}+C_1^{(i)}$ with $\As_m^{(i)}\in\Mat_2^{k_i}(S_F[1/p])$ by Lemma \ref{frob-h-mat-lem}. We conclude the existence of a matrix $Y\in\GL_2(R[1/p])$ such that $Y^{(i)}*_\varphi \As_m^{(i)}=\Af^{(i)}\in\Mat_2^{k_i}(\Sig_F)$ with $\Af^{(i)}\equiv A_1^{(i)}\equiv\Af'^{(i)}\equiv\Af_0^{(i)}\pmod\varpi$.
\end{proof}

\subsection{Preparatory Computations for Descent}\label{prep-comp-chap}

Recall that in Section \ref{Explicit-Kisin-mod-chap}, we were able to achieve an explicit description of Kisin modules $\Ms(\As)$ of finite height over $S_F=\Sig_F[\![E^p/p]\!]$ by describing their $i$-th partial Frobenius actions by the matrix $\As^{(i)}$ in terms of Types given by:
\begin{itemize}
	\item Type $\I$: \[\As^{(i)}=\begin{pmatrix}
		0 & E^{k_i}a_1^{(i)} \\ 1 & a_2^{(i)}\lambda_b^{g^{(i)}(\varphi)}
	\end{pmatrix};\]
	\item Type $\II$: \[\As^{(i)}=\begin{pmatrix}
		E^{k_i}a_1^{(i)} & 0 \\ a_2^{(i)}\lambda_b^{g^{(i)}(\varphi)} & 1
	\end{pmatrix}.\]
\end{itemize}

\noindent for some polynomials $g^{(i)}(\varphi)\in\Z[\varphi]$. Let us identify two disjoint subsets
\begin{align*}
	\Ss&\coloneqq\{i\in\Z/f\Z: \As^{(i)} \hspace{.1cm}\text{\normalfont is Type I}\hspace{.1cm}\}&\Ts&\coloneqq\{i\in\Z/f\Z: \As^{(i)} \hspace{.1cm}\text{\normalfont is Type II}\}.
\end{align*}

\noindent Hence, we have decomposed the set $\Z/f\Z=\Ss\sqcup\Ts$. We then have that $b=f$ if and only if the order $|\Ss|$ is even and $b=2f$ otherwise by Proposition \ref{det-base-change-existance-prop}. Additionally by Proposition \ref{red-fil-phi-mod-prop}, we see that $a_2^{(i)}\in\varpi\oh_F$ for all $i\in\Ts$ and $\prod_{i\in\Ss}a_2^{(i)}\in\varpi\oh_F$. We will continue to simplify notation by defining $(xy)^{(i)}=x^{(i)}y^{(i)}$.

\begin{defin}\label{height-c-def}
	For each $i\in\Z/f\Z$, let $c^{(i)}\ge1$ be the smallest integer such that $k_i\le c^{(i)}(p-2)$. We define $c_{max}=\max_i\{c^{(i)}\}$.
\end{defin}

As an implementation of the descent algorithm found in Section \ref{descent-alg-sec}, we would like to display the computations necessary to show that the Descent Assumptions \ref{descent-assump} hold for $\Ms(\As)$ when $c=c_{max}$ so long as $\nu_p(a_2^{(i)})$ is sufficiently large.

Looking at $\As^{(i)}$ regardless of Type, it is apparent that our only obstacle is the term $\lambda_b^{g^{(i)}(\varphi)}\in S_F^*$. Hence, our first move will be a row operation over $S_F[1/p]$ designed to strip off as many terms from each $\lambda_b^{g^{(i)}(\varphi)}$ as we can. Since $\lambda_b\in S_F^*$, let us write $\lambda_b^{g^{(i)}(\varphi)}=\sum_{j=0}^{\infty}\alpha^{(i)}_j (E^p/p)^j$ where each $\alpha^{(i)}_j\in\oh_F[u]$ of maximal degree $p-1$. For each $i\in\Z/f\Z$, define elements in $S_F[1/p]$ given by
\[x^{(i)}=-\left(\frac{a_2}{a_1}\right)^{(i)}\sum_{j=c^{(i)}}^{\infty}\alpha^{(i)}_j \frac{E^{jp-k_i}}{p^j}.\]

Set $X_1^{(i)}=\begin{psmallmatrix}
	1 & 0 \\ x^{(i)} & 1
\end{psmallmatrix}\in\GL_2(S_F[1/p])$ for all $i\in\Z/f\Z$ and observe that by setting $(X_1)*_\varphi(\As)=(\As_1)$ we will obtain
\[\As_1^{(i)}=\begin{dcases}
		\begin{pmatrix}
			E^{k_i}a_1^{(i)}\varphi(x^{(i-1)}) & E^{k_i}a_1^{(i)} \\ 1+\left(\sum_{j=0}^{c^{(i)}-1}a_2^{(i)}\alpha_j\frac{E^{jp}}{p^j}\right)\varphi(x^{(i-1)}) & \left(\sum_{j=0}^{c^{(i)}-1}a_2^{(i)}\alpha_j\frac{E^{jp}}{p^j}\right)
		\end{pmatrix} & \text{if	} i \in\Ss\\
		\begin{pmatrix}
			E^{k_i}a_1^{(i)} & 0 \\ \left(\sum_{j=0}^{c^{(i)}-1}a_2^{(i)}\alpha_j\frac{E^{jp}}{p^j}\right)+\varphi(x^{(i-1)}) & 1
		\end{pmatrix}& \text{if	} i \in\Ts. \\
	\end{dcases}\] 

\noindent We now require a lemma that makes sure  none of the $\varphi(x^{(i)})$-terms introduced as a result of $\varphi$-conjugation are non-integral or non-$I_{c_{max}}$ factors.

\begin{lem}\label{phi-admis-lem}
	Using the notations from above, suppose $\nu_p(a_2^{(i)})> c_{max}-c^{(i)}-1$. Then $\varphi(x^{(i)})\in \varpi p^{c_{max}-1}S_F\subset I_{c_{max}}$. 
	\begin{proof}
		Since each $\alpha_j\in \oh_F[u]$, then we need only prove that $a_2^{(i)}\varphi(E^{pj-k_i}/p^j)\in \varpi p^{c_{max}-1} S_F$ for each $j\ge c^{(i)}$. Let us write
		\[a_2^{(i)}\varphi\left(\frac{E^{pj-k_i}}{p^j}\right)=a_2^{(i)}\frac{\varphi(E)^{pj-k_i}}{p^j}=a_2^{(i)}p^{pj-k_i-j}\gamma^{pj-k_i}.\]
		Since $k_i\le c^{(i)}(p-2)$, then $pj-k_i-j\ge c^{(i)}$ for all $j\ge c^{(i)}$. Hence, so long as $\nu_p(a_2^{(i)})> c_{max}-c^{(i)}-1$ then we may conclude the desired result.
	\end{proof}
\end{lem}

With this, we observe  that the only non-integral and non-$I_{c_{max}}$ term is \[\sum_{j=0}^{c^{(i)}-1}a_2^{(i)}\alpha_j\frac{E^{jp}}{p^j}.\] If we assume that $\nu_p(a_2^{(0)})> c^{(i)}-1$ then this term becomes an element of $\varpi\Sig_F$ and we achieve satisfied descent.

\begin{prop}\label{dec-assump-large-val-prop}
	Suppose $\Ms(\As)$ is a rank-two Kisin module over $S_F$ with labeled heights $k_i\le c^{(i)}(p-2)$ with $c^{(i)}\ge1$ taken minimally and set $c_{max}=\max_i\{c^{(i)}\}$. If \[\nu_p(a_2^{(i)})> \max\{c_{max}-c^{(i)}-1,c^{(i)}-1\},\] then the Frobenius matrix $(\As)$ satisfies the Descent Assumptions \ref{descent-assump} with respect to $c=c_{max}$ and
	\begin{align*}
		\Af_0^{(i)}&=\begin{dcases}
			\begin{pmatrix}
				0 & E^{k_{i}}a_1^{(i)} \\ 1 & \left(\sum_{j=0}^{c^{(i)}-1}a_2^{(i)}\alpha_j\frac{E^{jp}}{p^j}\right)
			\end{pmatrix} & \text{if	} i\in\Ss\\
			\begin{pmatrix}
				E^{k_{i}}a_1^{(i)} & 0 \\ \left(\sum_{j=0}^{c^{(i)}-1}a_2^{(i)}\alpha_j\frac{E^{jp}}{p^j}\right) & 1
			\end{pmatrix}& \text{if	} i\in\Ts
		\end{dcases}
	\end{align*}
\end{prop}

\begin{rem}
	If $\max\{c_{max}-c^{(i)}-1,c^{(i)}-1\}=c^{(i)}-1$, then we may take $\nu_p(a_2^{(i)})=c^{(i)}-1$ to get the same $\Af_0^{(i)}$ as Proposition \ref{dec-assump-large-val-prop}, however the sum $\sum_{j=0}^{c^{(i)}-1}a_2^{(i)}\alpha_j\frac{E^{jp}}{p^j}$ will now only be an element of $\Sig_F$. Calculating reductions in these cases are much more complicated and must be done by hand so we exclude them here. See \cite{Guz2} for a treatment of this case when $c_{max}=2$.
\end{rem}

It this point, it is convenient to restate the above bounds involving $c^{(i)}$ to be in terms of $k_i$. Define $k_{max}=\max_i\{k_i\}$ and observe that
\[c^{(i)}-1=\Bigg\lfloor\frac{k_i-1}{p-2}\Bigg\rfloor\]
since $k_i-1<c^{(i)}(p-2)$. Next, we have that 
\[c_{max}-1-c^{(i)}=(c_{max}-1)-(c^{(i)}-1)-1=\Bigg\lfloor\frac{k_{max}-1}{p-2}\Bigg\rfloor-\Bigg\lfloor\frac{k_i-1}{p-2}\Bigg\rfloor-1.\]
Hence, Proposition \ref{dec-assump-large-val-prop} holds on the bounds 
\[\nu_p(a_2^{(i)})>\max\left\{\Bigg\lfloor\frac{k_i-1}{p-2}\Bigg\rfloor,\Bigg\lfloor\frac{k_{max}-1}{p-2}\Bigg\rfloor-\Bigg\lfloor\frac{k_i-1}{p-2}\Bigg\rfloor-1\right\}.\]

\section{Computing Explicit Reductions}\label{comp-red-chap}

\subsection{From Descent to Reductions}\label{mod-rep-sec}

By Proposition \ref{construct-kis-mod-prop}, we have constructed a finite height Kisin module $\Ms(\As)$ over $S_F=\Sig_F[\![E^p/p]\!]$ from $D(A)$ with $i$-th partial Frobenius $\As^{(i)}$ given in terms of one of two Types originating from Proposition \ref{det-base-change-existance-prop}: 
\begin{itemize}
	\item Type $\I$: \[\As^{(i)}=\begin{pmatrix}
		0 & E^{k_i}a_1^{(i)} \\ 1 & a_2^{(i)}\lambda_b^{g^{(i)}(\varphi)}
	\end{pmatrix};\]
	\item Type $\II$: \[\As^{(i)}=\begin{pmatrix}
		E^{k_i}a_1^{(i)} & 0 \\ a_2^{(i)}\lambda_b^{g^{(i)}(\varphi)} & 1
	\end{pmatrix}.\]
\end{itemize}

Supposing that we are able to show that the associated Frobenius matrix $(\As^{(i)})$ satisfies the Descent Assumptions \ref{descent-assump}, then the results of Theorem \ref{descent-thm} implies the existence of a base change $(Y^{(i)})$ over $R=\Sig_F[\![\varpi E/p]\!]$ such that for all $i\in\Z/f\Z$, we have
\[Y^{(i)}*_\varphi\As^{(i)}=\Af^{(i)}\in\Mat_2^{k_i}(\Sig_F).\]
In particular, we will have found a basis $\{\eta_1,\eta_2\}$ of $\Ms\otimes_{S_F[1/p]}R[1/p]$ such that the $\Sig_F$-submodule generated by $\{\eta_1,\eta_2\}$ has Frobenius action given by $(\Af^{(i)})$. We denote this integral submodule by $\Mf(\Af)$ and observe that it is a rank-two Kisin module over $\Sig_F$. In other words, $\Mf(\Af)$ is a descent of $\Ms(\As)$ to $\Sig_F$ in the sense of Definition \ref{descent-def}.

In Subsection \ref{BK-chap} we displayed how one may use a Kisin module such as $\Mf(\Af)$ to compute reductions of $V(A)$ \textit{if} we can canonically associated it to some $G_\infty$-stable $\oh_F$-lattice $T\subset V(A)$. The following proposition shows that this is indeed the case.

\begin{prop}\label{canon-assoc-prop}
	Using the notations from above, there exists a $G_\infty$-stable $\oh_F$ lattice $T\subset V(A)$ which is canonically associated to $\Mf(\Af)$.
	\begin{proof}
		Let us set $\Ms_R=\Mf(\Af)\otimes_{\Sig_F}R[1/p]$. To display the proposition, we need only justify why we may apply Proposition \ref{kis-lattice-ident-prop} to $\Ms_R$. By Theorem \ref{descent-thm} we have that 
		\[\Ms_R\cong \M_\Kis(\D_\cris^*(V(A)))\otimes_{\Lambda_F}R[1/p]\]
		so we must justify why we may fit $R[1/p]=\Sig_F[\![u,\frac{\varpi E}{p}]\!][1/p]$ inside of $\Lambda_{F,s}$ for some $1/p<s<1$. However, this is clear once we notice that $E/p\in\Lambda_{F,1}$ so that $\varpi E/p\in\Lambda_{F,s}$ for some $s<1$.
		
		To be completely accurate, we have not shown that $1/p<s$ and indeed the inequality may not be true. To recover the situation, notice that if $s<1/p$ then there exists an $1/p<s'$ such that $\Lambda_{F,s}\subset\Lambda_{F,s'}$ since $p^{-s'}<p^{-s}$. Hence, if $s<1/p$, we may extend scalars to $\Lambda_{F,s'}$ to apply Proposition \ref{kis-lattice-ident-prop}.
	\end{proof}
\end{prop}

By setting $\overline{\Mf(\Af)}\coloneqq\Mf(\Af)/\varpi\Mf(\Af)$, Proposition \ref{det-red-prop} says that the reduction $\overline{V(A)}\vert_{G_\infty}=(T/\varpi T)^\semi\vert_{G_\infty}$ is completely determined by the $k_F$-representation of $G_\infty$ given by
\[\overline{V(A)}|_{G_\infty}\cong \V_{k_F}^*\left(\overline{\Mf(\Af)}\left[\frac{1}{u}\right]\right).\]
The object $\overline{\Mf(\Af)}[\frac{1}{u}]$ is an \'{e}tale $\varphi$-module over $k_F(\!(u)\!)$ as described in Definition \ref{et-phi-mod-def}. In order to compute $\V_{k_F}^*(\overline{\Mf(\Af)}[\frac{1}{u}])$, recall that we have fixed a $p$-power compatible sequence $\pi=(\pi_0,\pi_1,\pi_2,\dots)\in\oh_{\C_p}^\flat$. We may embed
\[k_F[\![u]\!]\hookrightarrow \oh_{\C_p}^\flat\hspace{1cm}\text{\normalfont via}\hspace{1cm} u\mapsto \pi\]
with this embedding being naturally compatible with respect to both $\varphi$ and $G_\infty$. This embedding lifts to $k_F(\!(u)\!)\hookrightarrow \C_p^\flat$ so that we may write
\[\V_{k_F}^*\left(\overline{\Mf(\Af)}\left[\frac{1}{u}\right]\right)\coloneqq\Hom_{k_F(\!(u)\!),\varphi}\left(\overline{\Mf(\Af)}\left[\frac{1}{u}\right],\C_p^\flat\right).\]
Since we have computed $\overline{\Mf(\Af)}$ to be defined over $k_F[\![u]\!]$, then $\V_{k_F}^*(\overline{\Mf(\Af)}[\frac{1}{u}])$ is completely determined by
\[\V_{k_F}^*\left(\overline{\Mf(\Af)}\right)=\Hom_{k_F[\![u]\!],\varphi}\left(\overline{\Mf(\Af)},\oh_{\C_p}^\flat\right).\]
We summarize the above discussion in the following diagram:
\begin{figure}[H]
	\centering
\begin{tikzcd}
	V(A) \arrow[rr, "\D_\cris^*"] \arrow[d] &  & D(A) \arrow[rr, "\text{Prop \ref{construct-kis-mod-prop}}"] &  & \Ms(\As) \arrow[d, "\text{Thm \ref{descent-thm}}"]                               \\
	T|_{G_\infty} \arrow[d]              &  &                                     &  & \Mf(\Af) \arrow[d, "\mod\varpi"] \arrow[llll, leftrightarrow, "\text{Prop \ref{canon-assoc-prop}}"'] \\
	\overline{V(A)}|_{G_\infty}          &  &                                     &  & \overline{\Mf(\Af)} \arrow[llll, "\V_{k_F}^*"']                   
\end{tikzcd}
\end{figure}

When $K=\Qp$, one normally computes the Hom space of $\V_{k_F}^*$ by considering the action of the semi-simple Frobenius matrix on the basis elements. However, in our case of $K=\Qpf$, we instead deal with $f$-many partial Frobenius matrices, one in each idempotent piece. It is therefore convenient to reduce to looking at a single Frobenius matrix in the following construction from \cite[\S 2]{CDM14}.

As an \'{e}tale $\varphi$-module, we may perform the usual idempotent decomposition via the equivalence functor of Theorem \ref{equiv-et-phi-thm}. Hence, $\overline{\Mf(\Af)}=\prod_{i}\overline{\Mf(\Af)}^{(i)}$ so that each $\overline{\Mf(\Af)}^{(i)}$ is a $\varphi$-module with Frobenius action given by $\overline{\Af}^{(i)}\coloneqq \Af^{(i)}\pmod\varpi$. We may endow the $k_F[\![u]\!]$-module $\overline{\Mf(\Af)}^{(0)}$ with the structure of an \'{e}tale $\varphi^f$-module by defining its Frobenius action to be given by the product
\[\prod_{i\in\Z/f\Z}\Af^{(i)}=\overline{\Af}^{(0)}\varphi\left(\overline{\Af}^{(1)}\right)\varphi^2\left(\overline{\Af}^{(2)}\right)\cdots\varphi^{f-1}\left(\overline{\Af}^{(f-1)}\right).\] 
Such objects are defined just as in Definition \ref{et-phi-mod-def} after replacing $\varphi$-linearity with $\varphi^f$-linearity. The associated category would then be denoted by $\Mod_{k_F(\!(u)\!)}^{\varphi^f,\et}$ and there is an evident functor 
\begin{align*}
	\varepsilon_0&:\Mod_{k_F(\!(u)\!)}^{\varphi,\et}\rightarrow \Mod_{k_F(\!(u)\!)}^{\varphi^f,\et}&(M,\varphi)&\mapsto (M^{(0)}, \varphi^f).
\end{align*}

\noindent The following lemma shows that is suffices to compute the functor $\V_{k_F}^*$ by only looking at this associated \'{e}tale $\varphi^f$-module via $\varepsilon_0$.

\begin{lem}\label{red-in-first-piece-lem}
	Using the notation from above, let $(M^{(0)},\varphi^f)$ be the \'{e}tale $\varphi^f$-module over $k_F(\!(u)\!)$ associated to an \'{e}tale $\varphi$-module $(M,\varphi)$ over $k_F(\!(u)\!)$ via $\varepsilon_0$. 
	Then
	\[\V_{k_F}^*\left(M\right)\cong\Hom_{k_F(\!(u)\!),\varphi^f}\left(M^{(0)},\oh_{\C_p}^\flat\right).\]
	\begin{proof}
		As detailed in \cite[Thm 2.1.6]{CDM14} there is an equivalence of categories 
		\[\M_{k_F}^{(0)}:\Rep_{\cris/k_F}(G_\infty)\rightarrow \Mod_{k_F(\!(u)\!)}^{\varphi^f,\et}.\]
		with quasi-inverse functor $\V_{k_F}^{(0)}$. However, \cite[Prop 2.1.7]{CDM14} implies that for any $V\in\Rep_{\cris/k_F}(G_\infty)$, there is a natural isomorphism $\M_{k_F}^{{(0)}*}(V)\cong\M_{k_F}^{*}(V)^{(0)}$in $\Mod_{k_F(\!(u)\!)}^{\varphi^f,\et}$ where $\M_{k_F}^{*}(V)^{(0)}$ is the image of $\M_{k_F}^{*}(V)$ under the functor $\varepsilon_0$. It follows that $\V_{k_F}^*(M)\cong \V_{k_F}^{{(0)}*}(M^{(0)})$ where $\V_{k_F}^{{(0)}*}$ is defined as in \cite{CDM14} and the above discussion by $\V_{k_F}^{{(0)}*}(M^{(0)})= \Hom_{k_F(\!(u)\!),\varphi^f}(M^{(0)},\oh_{\C_p}^\flat)$.
	\end{proof}
\end{lem}

We are now ready to explorer the representation theory necessary to compute the reduction $\V_{k_F}^*(\overline{\Mf(\Af)})$. Let $I_K$ denote the inertial subgroup of $G_K$ and choose for any $r\ge1$ a $p^r-1$-th root of $\pi_K=-p$ denoted $\xi_r=(-p)^{1/p^r-1}\in \overline{K}$. This choice of $\xi_r$ defines a character $\omega_{\xi_r}:I_K\rightarrow \oh_K^*$. By composing with our fixed embedding $\tau_0:\oh_K\hookrightarrow\oh_F$, we give rise to a \textit{fundamental character of niveau $r$}
\[\widetilde{\omega}_{\xi_r}:I_K\rightarrow\oh_F^*.\]
The residual character of $\widetilde{\omega}_{\xi_r}$, denoted $\omega_r$, is then a \textit{fundamental character of level $r$} such that the its action on any $g\in G_\infty$ is given by $\omega_r(g)=g(\xi_r)/\xi_r$. The following proposition does the job of computing the Hom-space described in Lemma \ref{red-in-first-piece-lem}. We follow the ideas found in the proof of \cite[Prop 5.1]{Liu21} which itself is a special case of \cite[Prop 3.1.2]{LHL19}.

\begin{prop}\label{compute-Hom-prop}
	Let $\overline{\Mf}=\prod_{i}\overline{\Mf}^{(i)}$ be a finite height Kisin module over $k_F[\![u]\!]$ with basis $\{e_1^{(i)},\dots,e_d^{(i)}\}_{i\in\Z/f\Z}$. Suppose that $\varphi^f(e_j^{(0)})=u^{a_j}e_{j+1}^{(0)}$ with $e_{d+1}^{(0)}=e_1^{(0)}$, then 
	\[\V_{k_F}^*(\overline{\Mf})|_{G_\infty}=\ind_{G_{\Q_{p^{df}}}}^{G_{\Q_{p^f}}}\omega_{df}^{\sum_{j=1}^{d}p^{d-j}a_j}.\]
	\begin{proof}	
		As discussed above, we may view $\overline{\Mf}$ as an \'{e}tale $\varphi$-module over $k_F[\![u]\!]$ and by Lemma \ref{red-in-first-piece-lem} we have reduced to computing 
		\[\V_{k_F}^*\left(\overline{\Mf}\right)\cong\Hom_{k_F[\![u]\!],\varphi^f}\left(\overline{\Mf}^{(0)},\oh_{\C_p}^\flat\right).\]
		Consider a general homomorphism $X=(x_1,\dots,x_d)\in \Hom_{k_F[\![u]\!],\varphi^f}(\overline{\Mf}^{(0)},\oh_{\C_p}^\flat)$. We may see that  $\varphi^f$-compatibility necessarily implies that $\varphi^f(x_j)=\pi^{a_j}x_{j+1}$ via the embedding $k_F[\![u]\!]\hookrightarrow \oh_{\C_p}^\flat$ sending $u\mapsto \pi$ where $\pi$ is a $p$-power compatible sequence in $\oh_{\C_p}^\flat$ starting at $\pi_K=-p$. It follows that 
		\[(\varphi^f)^d (x_j)=\pi^{\sum_{\ell=1}^{d}p^{d-\ell}a_\ell}(x_j)\]
		and hence, that $x_1$ is a solution to the equation
		\[x^{p^{df}}=\pi^{\sum_{j=1}^{d}p^{d-j}a_j}(x).\]
		Recall that $\xi_r=(-p)^{1/p^r-1}$ for any $r\ge 1$. Let us pick $\chi=(\chi_j)_{j\ge0}\in\oh_{\C_p}^\flat$ such that $\chi_0=\xi_{df}$ and $\chi^{p^{df}-1}=\pi$. By setting 
		\[x_1=\chi^{\sum_{j=1}^{d}p^{d-j}a_j}\]
		then $x_1$ is a solution to the above equation and all other roots may be attained by multiplication with a $p^{df}-1$-th root of unity in $\overline{\F}_p^*$. Hence, all other $x_j$ are determined by our choice of $x_1$ so that $X$ is determined by $x_1$. As a result, for any $g\in G_\infty=\Gal(K_\infty/K)$, the action of $g(X)$ is determined by the action of $g(x_1)$ so that if we set $H$ to be the splitting field of $Y^{p^{df}-1}+p$ in $\overline{\Q}_p$, then the $G_\infty$ action on $\V_{k_F}^*(\overline{\Mf})$ factors through $\Gal(H/K)$. 
		
		Since $\Gal(H/K)$ is finite, then we have a short exact sequence
		\[0\rightarrow I_{H/K}\rightarrow \Gal(H/K)\rightarrow \Gal(k_H/\Fpf)\rightarrow 0\]
		where $I_{H/K}$ denotes the inertia subgroup of $\Gal(H/K)$. As $H$ is totally ramified over $K$, then the Schur–Zassenhaus theorem \cite[Theorem 7.41]{Rot95} implies that $\Gal(H/K)$ is the semi-direct product $\Gal(H/K)\cong\Gal(k_H/\Fpf)\rtimes I_{H/K}$. Hence, upon restricting to inertia, we see that for any $g\in I_{H/K}$ we may write
		\[g(X)=\omega_{df}^{\sum_{j=0}^{d}p^{d-j}a_j}(g)X\]
		due to the fact that $g(\chi)=\omega_{df}(g)\chi$ by the definition of $\chi$ and the fundamental character.
		
		The proposition then follows after observing the following fact from \cite[Lemma 6.1]{Dou10}. Let $W=\omega_{df}^{\sum_{j=0}^{d}p^{d-j}a_j}$ be interpreted as the restriction of $\V_{k_F}^*(\overline{\Mf})$ to $I_{H/K}$. By the definition of $H$, we may view $W$ as a representation of $\Gal(\Q_{p^\infty}/\Q_{p^{df}})$ which has finite index $d$ in $G_\infty$. Since $\V_{k_F}^*(\overline{\Mf})$ is irreducible then Frobenius reciprocity will imply an isomorphism
		\[\V_{k_F}^*(\overline{\Mf})|_{G_\infty}\cong \ind_{G_{\Q_{p^{df}}}}^{G_{\Q_{p^f}}}\omega_{df}^{\sum_{j=0}^{d}p^{d-j}a_j}.\]
	\end{proof}
\end{prop}

One complication may stand out to the reader at this point. That is, how are we to compute the product 
\[\prod_{i\in\Z/f\Z}\Af^{(i)}=\overline{\Af}^{(0)}\varphi\left(\overline{\Af}^{(1)}\right)\varphi^2\left(\overline{\Af}^{(2)}\right)\cdots\varphi^{f-1}\left(\overline{\Af}^{(f-1)}\right)\] 
in an explicit way so as to use Proposition \ref{compute-Hom-prop} when $f$ is allowed to be arbitrarily large? Indeed, this is impractical but let us introduce some notation to recover the situation somewhat. 

Let $I=\ident_2$ denote the identity matrix and let $S=\begin{psmallmatrix}
	0 & 1 \\ 1 & 0
\end{psmallmatrix}$. For a pair $\lambda_i=(n_i,m_i)\in\Z^2$, we will set $u^{\lambda_i}=\Diag(u^{n_i},u^{m_i})$. In the coming sections, we will describe a $k_F[\![u]\!]$-basis such that each $\overline{\Af}^{(i)}$ is either $Iu^{\lambda_i}$ or $Su^{\lambda_i}$ up to scalars in $k_F$. We will then summarize this data in the form of a map $\overline{\Mf(\Af)}\mapsto \mu=(\mu_i)_{i\in\Z/f\Z}$ where each $\mu_i=(M_i,\lambda_i)$ with $M_i\in\{I,S\}$ and $\lambda_i$ is as defined above. Let us also decompose $i\in\Z/f\Z$ into two disjoint subsets $\Z/f\Z=\Is\sqcup\Ps$ where $\Is$ denotes the set of $i$ such that $A_i=I$ and $\Ps$ denotes the set of $i$ with $A_i=S$. In practice, $A_i$ tells us if the partial Frobenius $\overline{\Af}^{(i)}$ is reducible or irreducible and $\lambda_i$ encodes its elementary divisors in $k_F[\![u]\!]$. 

\begin{defin}\label{red-data-def}
	We will call the $f$-tuple of pairs $\mu=(\mu_i)=(M_i,\lambda_i)_i$ derived from $\overline{\Mf(\Af)}$ to be its \textit{reduction data}.
\end{defin}

\noindent As the name suggests, this data suffices to compute the reduction $\overline{V(A)}|_{G_\infty}$ explicitly given specific data, up to an unramified twist in the irreducible case or restriction to inertia in the reducible case. The following Proposition lays out an algorithm to compute the reduction given such data.

\begin{prop}\label{compute-hom-with-data-prop}
	Using the notation from above, let $\overline{\Mf(\Af)}$ be a finite height, rank-two Kisin module with reduction data $\mu=(\mu_i)$ canonically associated to $T\subset V(A)$. Then there exists $v_i,w_i\in\Z$ such that one of two cases hold:
	\begin{enumerate}
		\item If the order of the set $|\Ps|$ for $\mu$ is even or zero, then we have a reducible reduction
		\[\overline{V(A)}|_{I_K}=\omega_f^{\sum_{j=0}^{f-1} p^{j}v_j}\oplus\omega_f^{\sum_{j=0}^{f-1} p^{j}w_j};\]
		\item If the order of the set $|\Ps|$ is odd, set $t=p\sum_{j=0}^{f-1} p^{j}w_j+\sum_{j=0}^{f-1} p^{j}v_j$. 
		\begin{enumerate}
			\item[(i)] When $p^f-1\nmid t$ then we have an irreducible reduction up to unramified twist
			\[\overline{V(A)}|_{G_\infty}=\ind^{G_{\Qpf}}_{G_{\Q_{p^{2f}}}}\left(\omega_{2f}^{t}\right).\]
			
			\item[(ii)] When $p^f-1\mid t$ then we have a reducible reduction up to restriction to inertia
			\[\overline{V(A)}|_{I_K}=\omega_f^{\frac{t}{p^f-1}}\oplus\omega_f^{\frac{t}{p^f-1}}.\]
		\end{enumerate}
	\end{enumerate}
	\begin{proof}
		We will construct the required $v_i$ and $w_i$ from the powers of the elementary divisors $n_i$ and $m_i$ of the reduction data $\mu$. A sum of these elements will then form the powers of $u$ in the product $\prod\overline{\Af}^{(i)}$ so that we may apply Proposition \ref{compute-Hom-prop} to get a reduction. Since $\Ps$ denotes the set of $i\in\Z/f\Z$ associated to anti-diagonal matrices, then it is easy to see that the order $|\Ps|$ being odd will result in an irreducible product and an even order will result in a reducible product.
		
		Let us first suppose that the set $\Ps=\emptyset$ so that the reduction data is given by $\mu_i=(I,\lambda_i)$ for all $i\in\Z/f\Z$. Hence, the product is easily seen to be
		\[\prod_{i\in\Z/f\Z}\overline{\Af}^{(i)}=\begin{pmatrix}
			u^{\sum_{j=0}^{f-1} p^jn_j} & 0 \\ 0 & u^{\sum_{j=0}^{f-1} p^jm_j}
		\end{pmatrix}.\]
		We may then set $v_i=n_i$ and $w_i=m_i$ for all $i\in\Z/f\Z$ and apply Proposition \ref{compute-Hom-prop} to get a reduction in the form of $(a)$.
		
		Now let us tackle the more complicated case of $\Ps\neq\emptyset$. We may enumerate the set $\Ps$ by $\Ps=\{r_0,r_1,\dots, r_d\}$ where $r_{j-1}<r_j$ and $0\le d\le f-1$. For $\ell\ge0$, let us assign:
		\begin{itemize}
			\item For $0\le i<r_0$ and $r_0\neq 0$, set $v_i=n_i$ and $w_i=m_i$;
			\item For $r_{0}\le i<r_{1}$, set $v_i=m_i$ and $w_i=n_i$;
			\item Continue this process in that for $\ell$ odd we set $v_i=n_i$ and $w_i=m_i$ for $r_{\ell}\le i<r_{\ell+1}$ and $v_i=m_i$ and $w_i=n_i$ for $r_{\ell+1}\le i<r_{\ell+2}$. 
			\item When we reach $r_d$, repeat the pattern for $r_d\le i<f-1$.
		\end{itemize}
		
		\noindent The properties of matrix multiplication between diagonal and anti-diagonal matrices means that we will have one of two cases depending on the order $|\Ps|$. Namely, if $|\Ps|$ is even then $d$ is odd and the product will be the reducible matrix
		\[\prod_{i\in\Z/f\Z}\overline{\Af}^{(i)}=\begin{pmatrix}
			u^{\sum_{j=0}^{f-1} p^jv_j} & 0 \\ 0 & u^{\sum_{j=0}^{f-1} p^jw_j}
		\end{pmatrix}.\]
		We may apply Proposition \ref{compute-Hom-prop} to get a reduction in the form of $(a)$. Finally, if $|\Ps|$ is odd then $d$ is even and the product will be an irreducible matrix in the form
		 		\[\prod_{i\in\Z/f\Z}\overline{\Af}^{(i)}=\begin{pmatrix}0 &
		 	u^{\sum_{j=0}^{f-1} p^jv_j} \\ u^{\sum_{j=0}^{f-1} p^jw_j} & 0
		 \end{pmatrix}.\]
		We may then apply Proposition \ref{compute-Hom-prop} to get a reduction in the form of $(b)$ after recalling that $\omega_{2f}^{p^f-1}=\omega_f$.
	\end{proof}
\end{prop}

\subsection{Reductions for Large Valuations}\label{red-large-val-sec}

Let $\Ms(\As)$ be a Kisin module over $S_F$ as constructed in Section \ref{Explicit-Kisin-mod-chap}. The results of Proposition \ref{dec-assump-large-val-prop} combined with Theorem \ref{descent-thm} allow use to identify a descent of $\Ms(\As)$ to $\Sig_F$ denoted $\Mf(\Af)$ so long as 
\[\nu_p(a_2^{(i)})>\max\left\{\Bigg\lfloor\frac{k_i-1}{p-2}\Bigg\rfloor,\Bigg\lfloor\frac{k_{max}-1}{p-2}\Bigg\rfloor-\Bigg\lfloor\frac{k_i-1}{p-2}\Bigg\rfloor-1\right\}.\]
The reduction of which $\overline{\Mf(\Af)}$ necessarily has partial Frobenius matrices given by
\begin{align*}
	\overline{\Af}^{(i)}&=\begin{dcases}
			\begin{pmatrix}
				0 & u^{k_{i}}\overline{a}_1^{(i)} \\ 1 & 0
			\end{pmatrix} & \text{if	} i\in\Ss\\
			\begin{pmatrix}
				u^{k_{i}}\overline{a}_1^{(i)} & 0 \\ 0 & 1
			\end{pmatrix}& \text{if	} i\in\Ts 
		\end{dcases}
\end{align*}

\noindent due to the fact that $E=u+p$. Hence, the Kisin module $\overline{\Mf(\Af)}$ over $k_F\pu$ will have reduction data given by 
\begin{align*}
	\mu_i&=\begin{dcases}
	(S, (0, k_i)) & \text{if	} i\in\Ss \\
	(I, (k_i, 0)) & \text{if	} i\in\Ts.
\end{dcases}
\end{align*}
By the results of Section \ref{mod-rep-sec}, we may use Proposition \ref{compute-hom-with-data-prop}, where $v_i,w_i\in\{k_i,0\}$, to explicitly compute reductions of $V(A)$ under the aforementioned restrictions on $\nu_p(a_2^{(i)})$. With this, we arrive at our main result.

\begin{thm}\label{large-val-red-thm}
	Let $V\in\Rep_{\cris/F}^{\kvec}(G_K)$ be an irreducible, two dimensional crystalline representation of $G_K$ with labeled Hodge-Tate weights $\{k_i,0\}_{i\in\Z/f\Z}$ where $k_i>0$. Then $V\cong V(A)$ and if $(A^{(i)})_{i\in\Z/f\Z}$ is such that 
	\[\nu_p(a_2^{(i)})>\max\left\{\Bigg\lfloor\frac{k_i-1}{p-2}\Bigg\rfloor,\Bigg\lfloor\frac{k_{max}-1}{p-2}\Bigg\rfloor-\Bigg\lfloor\frac{k_i-1}{p-2}\Bigg\rfloor-1\right\},\]
	then there exists $v_i,w_i\in\{k_i,0\}$ such that:
	\begin{enumerate}
		\item If $|\Ss|$ is even, then 
		\[\overline{V}|_{I_K}=\omega_f^{\sum_{j=0}^{f-1} p^{j}v_j}\oplus\omega_f^{\sum_{j=0}^{f-1} p^{j}w_j};\]
		\item If $|\Ss|$ is odd, set $t=p\sum_{j=0}^{f-1} p^{j}w_j+\sum_{j=0}^{f-1} p^{j}v_j$, 
		\begin{enumerate}
			\item[(i)] When $p^f-1\nmid t$ then 
			\[\overline{V}|_{G_\infty}=\ind^{G_{\Qpf}}_{G_{\Q_{p^{2f}}}}\left(\omega_{2f}^{t}\right);\]
			
			\item[(ii)] When $p^f-1\mid t$ then 
			\[\overline{V}|_{I_K}=\omega_f^{\frac{t}{p^f-1}}\oplus\omega_f^{\frac{t}{p^f-1}}.\]
		\end{enumerate}
	\end{enumerate}
	\begin{proof}
		The fact that $V\cong V(A)$ for some $(A^{(i)})$ comes from Theorem \ref{model-irred-cris-reps-thm}. Proposition \ref{dec-assump-large-val-prop} provides the necessary bound on $a_2^{(i)}$ with which the descent $\Mf(\Af)$ described above exists. The theorem then follows by applying Proposition \ref{compute-hom-with-data-prop} to the descent data $\mu_i$ of $\overline{\Mf(\Af)}$ as described previously.
	\end{proof}
\end{thm}


\printbibliography[title = {References}]


\end{document}